\documentclass[a4paper,11pt]{amsart}
\usepackage{color,graphicx}
\title[Principal eigenvalue of a linear elliptic operator with small/large diffusion]
{Asymptotic behavior of the principal eigenvalue of a linear second order
elliptic operator with small/large diffusion coefficient and its application}
\author[R. Peng, G. Zhang and M. Zhou%Y. Kaneko and
]{%Yuki Kaneko$^\dag$ and
Rui Peng${}^\dag$, Guanghui Zhang${}^{\S,*}$ and Maolin Zhou${}^\ddag$}
\thanks{${}^\dag$School of Mathematics and Statistics, Jiangsu Normal University, Xuzhou, 221116, Jiangsu Province, China. (Email: {\tt pengrui\,$\b{}$\,seu@163.com})}
\thanks{${}^\S$School of Mathematics and Statistics, Huazhong University of Science and Technology, Wuhan,
430074, China.(Email: {\tt guanghuizhang@hust.edu.cn})}
\thanks{${}^*$Hubei Key Laboratory of Engineering Modeling and Scientific Computing, Huazhong University of Science and Technology, Wuhan 430074, China.}
\thanks{${}^\ddag$School of Science and Technology, University of New England, Armidale, NSW, 2351, Australia. (Email: {\tt mzhou6@une.edu.au})}
\thanks{R. Peng was supported by NSF of China (No. 11671175, 11571200), the Priority Academic Program Development of Jiangsu Higher Education Institutions, Top-notch Academic Programs Project of Jiangsu Higher Education Institutions (No. PPZY2015A013) and Qing Lan Project of Jiangsu Province, and G. Zhang was supported by NSF of China (No. 11501225) and the Fundamental Research Funds for the Central Universities (No. 5003011008).}
\date{\today}
\usepackage{amsmath, enumerate, color, mathrsfs, cases}

\makeatletter

\@addtoreset{equation}{section}
\makeatother
\theoremstyle{definition}

\theoremstyle{plain}
\newtheorem{thm}{Theorem}[section]

\newtheorem{proposition}{Proposition}[section]
\newtheorem{cor}[proposition]{Corollary}
\newtheorem{lemma}{Lemma}[section]

\theoremstyle{definition}
\newtheorem{remark}{Remark}[section]
\begin{document}

\begin{abstract} In this article, we are concerned with the following eigenvalue problem of a linear second order elliptic operator:
 \begin{equation}
  \nonumber
  -D\Delta \phi -2\alpha\nabla m(x)\cdot \nabla\phi+V(x)\phi=\lambda\phi\ \ \hbox{ in }\Omega,
\end{equation}
complemented by a general boundary condition including Dirichlet boundary condition and Robin boundary condition:
 $$
 \frac{\partial\phi}{\partial n}+\beta(x)\phi=0 \ \ \hbox{ on  }\partial\Omega,
 $$
where $\beta\in C(\partial\Omega)$ allows to be positive, sign-changing or negative, and $n(x)$ is the unit exterior normal to $\partial\Omega$ at $x$. The domain $\Omega\subset\mathbb{R}^N$ is bounded and smooth, the constants $D>0$ and $\alpha>0$ are, respectively, the diffusive and advection coefficients, and $m\in C^2(\bar\Omega),\,V\in C(\bar\Omega)$ are given functions.

We aim to investigate the asymptotic behavior of the principal eigenvalue of the above eigenvalue problem as the diffusive coefficient $D\to0$ or $D\to\infty$. Our results, together with those of \cite{CL2,DF,Fr} where the Nuemann boundary case (i.e., $\beta=0$ on $\partial\Omega$) and Dirichlet boundary case were studied, reveal the important effect of advection and boundary conditions on the asymptotic behavior of the principal eigenvalue. We also apply our results to a reaction-diffusion-advection equation which is used to describe the evolution of a single species living in a heterogeneous stream environment and show some interesting behaviors of the species persistence and extinction caused by the buffer zone and small/large diffusion rate.
\end{abstract}

\maketitle

\section{Introduction}

In \cite{CL2}, Chen and Lou studied the following eigenvalue problem of
a linear second order elliptic operator with Neumann boundary condition:
\begin{equation}
  \label{eq:N}
\begin{cases}
  -D\Delta \phi -2\alpha\nabla m(x)\cdot \nabla\phi+V(x)\phi=\lambda\phi &\hbox{ in }\Omega,\\
  \frac{\partial\phi}{\partial n}=0 &\hbox{ on }\partial\Omega.
\end{cases}
\end{equation}
From now on, unless otherwise stated, we assume that $\Omega$ is a bounded smooth domain in $\mathbb{R}^N$, the
constants $D>0$ and $\alpha>0$ stand for the diffusive and advection
coefficients, respectively, $m\in C^2(\bar\Omega),\,V\in C(\bar\Omega)$ are given functions, and $n(x)$ is the unit exterior normal to $\partial\Omega$ at $x$.

Given $D>0,\,\alpha,\,m$ and $V$, it is well known that \eqref{eq:N} admits a smallest eigenvalue
(also called as {\it principal eigenvalue}), denoted by $\lambda(D)$, which corresponds to a positive
eigenfunction (called as {\it principal eigenfunction}). The principal eigenvalue is a basic concept in
the field of reaction-diffusion equations, and it
usually plays a vital role in the study of a nonlinear reaction-diffusion equation. In particular, the limiting behavior of
$\lambda(D)$ as $D\to\infty$ or $D\to0$ is important in order to obtain a good understanding of the
qualitative behavior of a reaction-diffusion equation under consideration.

For such a purpose, among several other ones, Chen and Lou established the following important result;
see \cite[Theorems 1.2 and 8.1]{CL2}.

\begin{thm}\label{theorem-N} Assume that ${\rm{det}}(D^2m(x))\not=0$ and $n(x)$ is an eigenvector of
$D^2m(x)$ for all $x\in\{x\in\partial\Omega:\ |\nabla m(x)|=0\}$. Then it holds
\begin{equation}\label{N-f}
\lim_{D\to 0}\lambda(D)=\min_{x\in\Sigma_{1}^*\cup\Sigma_{2}^*}\Big\{V(x)+\alpha\sum_{i=1}^N(|\kappa_i(x)|+\kappa_i(x))\Big\},
\end{equation}
where
 $$
 \Sigma_1^*=\{x\in\bar\Omega:\ |\nabla  m(x)|=0\},\ \ \Sigma_2^*=\{x\in\partial\Omega:\ |\nabla
  m(x)|=n(x)\cdot \nabla m(x)>0\},
  $$
and when $x\in\Sigma_1^*$, $\kappa_1(x), \kappa_2(x),\cdots,\kappa_N$ are the eigenvalues of
$D^2m(x)$, and when $x\in\Sigma_2^*$, $\kappa_N=0$, and $\kappa_1(x),\cdots,\kappa_{N-1}(x)$ are the eigenvalues of
$D^2m_{\partial\Omega}(x)$. Here, $m_{\partial\Omega}(x)$ is the restriction of $m(x)$ on $\partial\Omega$.

\end{thm}

\begin{remark}\label{r-N} It is worthwhile to mention that the formula \eqref{N-f} may not be true provided that the technical assumption imposed on $m$ on the set $\Sigma_{1}^*\cap\partial\Omega$ in Theorem \ref{theorem-N} fails for the space dimension $N\geq2$; see \cite[Remarks 8.2, 8.3]{CL2}. However, one can easily check the proof of \cite[Theorem 8.1]{CL2} to see that Theorem \ref{theorem-N} remains true without such an assumption when $N=1$ (that is, $\Omega$ is a bounded open interval).

\end{remark}

In the two companion papers \cite{DF,Fr}, Friedman and his coauthors considered the following Dirichlet eigenvalue problem
 \begin{equation}
  \label{eq:1-D}
\begin{cases}
  -D\sum_{i,j=1}^Na_{ij}(x)\frac{\partial^2\phi}{\partial x_i\partial x_j}-\sum_{i=1}^Nb_{i}(x)\frac{\partial\phi}{\partial x_i}=\lambda\phi &\hbox{ in }\Omega,\\
  \phi=0 &\hbox{ on }\partial\Omega.
\end{cases}
\end{equation}
Assume that the matrix $(a_{ij}(x))$ is real, positive definite and symmetric for any $x\in\bar\Omega$, and $a_{ij},\,b_i\, (i=1,\cdots, N)$ are H$\ddot{o}$lder continuous functions on $\bar\Omega$. Denote $b(x)=(b_1(x),\cdots,b_N(x))$.

The main results of \cite{DF,Fr} can be collected as follows.

\begin{thm}\label{theorem1-D} Let $\lambda(D)$ be the principal eigenvalue of \eqref{eq:1-D}. The following assertions hold.
\begin{itemize}
\item[\rm{(i)}] If $b(x)\cdot n(x)<0$ for any $x\in\partial\Omega$, then
 $\lim_{D\to 0}\lambda(D)=0$ (exponentially fast).

\item[\rm{(ii)}] If there exists a function $\omega\in C^1(\bar\Omega)$ such that $b(x)\cdot{\rm grad}\,\omega>0$ for any $ x\in\bar\Omega$, then $\lim_{D\to 0}\lambda(D)=+\infty$.

\item[\rm{(iii)}] If there exists a point $x_0\in\Omega$ such that $b(x)=O(|x-x_0|^\nu)$ as $x\to x_0$ for some $\nu\geq0$, then $\limsup_{D\to 0}\lambda(D)\leq cD^{(\nu-1)/(\nu+1)}$ for some positive constant $c$.

\item[\rm{(iv)}] If there exist a point $x_0\in\Omega$ and a function $\omega\in C^2(\bar\Omega)$ such that
the matrix $(\omega_{ij})(x_0)$ is positive definite and $b(x)\,\cdot\,{\rm grad}\,\omega\geq \varrho|x-x_0|^{\nu+1},\,\forall x\in\bar\Omega$ for some constants $\varrho>0,\,1>\nu>0$, and $\liminf_{x\to x_0}{\rm div}\, b(x)>-\infty$, then $\liminf_{D\to 0}\lambda(D)\geq cD^{(\nu-1)/(\nu+1)}$ for some positive constant $c$.

\end{itemize}

\end{thm}

\begin{remark}\label{r-D1} When the space dimension $N=1$ and we take $\Omega=(0,1)$ without loss of generality, for the principal eigenvalue $\lambda(D)$ of problem \eqref{eq:1-D}, Theorem \ref{theorem1-D} concludes that $\lim_{D\to 0}\lambda(D)=+\infty$ if one of the conditions is fulfilled:
\begin{itemize}
\item[(i)]  $|b|>0$ on $[0,1]$.

\item[(ii)]  If $b<0$ in $[0,x_0)$, $b>0$ in $(x_0,1]$, and $|x-x_0|^{\nu}/\sigma\leq|b(x)|\leq\sigma|x-x_0|^{\nu},\,\forall x\in[0,1]$
for some $x_0\in(0,1)$ and constants $\sigma>1,\,1>\nu>0$.

\end{itemize}

\end{remark}

The objective of this paper is the asymptotic behavior of the principal eigenvalue with respect to small or large diffusion coefficient, of a class of eigenvalue problems under certain boundary conditions including the Dirichlet boundary condition and Robin boundary condition.
The existence and uniqueness of the principal eigenvalue of the eigenvalue problem
to be treated in the paper as well as its variational characterization are standard facts;
see, for instance, \cite[Chapter 2]{Du}.

We first investigate the Dirichlet eigenvalue problem
\begin{equation}
  \label{eq:1}
\begin{cases}
  -D\Delta \phi -2\alpha\nabla m(x)\cdot \nabla\phi+V(x)\phi=\lambda\phi &\hbox{ in }\Omega,\\
  \phi=0 &\hbox{ on }\partial\Omega.
\end{cases}
\end{equation}

Denote
 $$
 \Sigma_1=\{x\in \Omega:\ |\nabla m(x)|=0\},\ \ \Sigma_2=\{x\in\partial\Omega:\ |\nabla m(x)|=0\}.
  $$
Concerning \eqref{eq:1}, our result reads as follows.

\begin{thm}\label{theorem1} Let $\lambda(D)$ be the principal eigenvalue of \eqref{eq:1}. The following assertions hold.
\begin{itemize}
\item[\rm{(i)}] If $\Sigma_{1}\cup\Sigma_{2}=\emptyset$, then
 $\lim_{D\to 0}\lambda(D)=+\infty$.

\item[\rm{(ii)}] If $\Sigma_{1}\cup\Sigma_{2}\neq\emptyset$, and assume that for all
  $x\in\Sigma_{2}$, $n(x)$ is an eigenvector of $D^{2}m(x)$ with the corresponding eigenvalue $\kappa_{N}= 0$, then
\[
\lim_{D\to 0}\lambda(D)=\min_{x\in\Sigma_{1}\cup\Sigma_{2}}\Big\{V(x)+\alpha\sum_{i=1}^N(|\kappa_i(x)|+\kappa_i(x))\Big\},
\]
here $\kappa_1(x), \kappa_2(x),\cdots,\kappa_N$ are the eigenvalues of $D^2m(x)$.

\end{itemize}

\end{thm}

We would like to make the following comments on Theorem \ref{theorem1}.
\begin{remark}\label{r-D-D}
\begin{itemize}

\item[(i)] We want to stress that in Theorem \ref{theorem-N}, we take $\kappa_N=0$, and $\kappa_1(x),\cdots,\kappa_{N-1}(x)$ are the $N-1$ eigenvalues of $D^2m_{\partial\Omega}(x)$ for any $x\in\Sigma_{2}^*$; while in Theorem \ref{theorem1}(ii), for any $x\in\Sigma_{2}$, we assume that $(0,n(x))$ is an eigenpair of $D^2m(x)$, and thus $\kappa_1(x), \kappa_2(x),\cdots,\kappa_N=0$ are all the eigenvalues of $D^2m(x)$.

\item[(ii)] Clearly, Theorem \ref{theorem1-D}(ii) covers Theorem \ref{theorem1}(i); we shall provide an elementary proof for Theorem \ref{theorem1}(i); the proof of Theorem \ref{theorem1-D}(ii) is much more involved via probabilistic inequalities (see \cite{Fr}).

\item[(iii)] Theorem \ref{theorem1}(ii) shows that $\lim_{D\to 0}\lambda(D)$ must be finite once $\Sigma_{1}\cup\Sigma_{2}\neq\emptyset$ under the assumption of $m\in C^2(\bar\Omega)$. However, Theorem \ref{theorem1-D}(iv) tells us that
    $\lim_{D\to 0}\lambda(D)$ may be positive infinity even if $\Sigma_{1}\cup\Sigma_{2}\neq\emptyset$ provided that $m\in C^{1+\nu}(\bar\Omega)$ for some $0<\nu<1$; also see Remark \ref{r-D1}. Therefore, this implies that the smoothness of the advection $m$ is also vital in determining $\lim_{D\to 0}\lambda(D)$.

\end{itemize}

\end{remark}

We next consider the eigenvalue problem equipped with Robin boundary condition:
\begin{equation}
  \label{eq:3.1}
\begin{cases}
  -D\Delta \phi -2\alpha\nabla m(x)\cdot \nabla\phi+V(x)\phi=\mu\phi &\hbox{ in }\Omega,\\
  \frac{\partial\phi}{\partial n}+k\beta(x)\phi=0 &\hbox{ on
  }\partial\Omega,
\end{cases}
\end{equation}
where $\beta\in C(\partial{\Omega})$ is a given function and $k$ is a
nonnegative constant. Note that $\beta$ allows to change sign or be positive, or be negative over $\partial\Omega$;
one may refer to \cite{FK,LOS,PP} and the references therein for related background and research, especially when
$\beta$ is negative.

Let $\Sigma_1$ and $\Sigma_2$ be as before, and set
 $$
 \Sigma_3=\{x\in\partial\Omega:\ |\nabla m(x)|=\nabla m(x)\cdot n(x)>0\}.
 $$
Concerning the eigenvalue problem \eqref{eq:3.1}, our result can be stated as follows.

\begin{thm}\label{th1.2} Let $\lambda(D)$ be the principal eigenvalue of \eqref{eq:3.1}.
  Assume that $\mathrm{det}(D^{2}m(x))\neq 0$ and $n(x)$ is an eigenvector of
  $D^{2}m(x)$ for all $x\in\Sigma_{2}$. Then it holds
\begin{align*}
  \lim_{D\to 0}\lambda(D)&=\min\Big\{\min_{x\in\Sigma_{1}\cup\Sigma_{2}}\big\{V(x)+\alpha\sum_{i=1}^N(|\kappa_i(x)|+\kappa_i(x))\big\},\\
                              &\ \ \ \ \min_{x\in\Sigma_{3}}\big\{V(x)+\alpha\sum_{i=1}^N(|\kappa_i(x)|+\kappa_i(x))+2\alpha k \beta(x_{0})|\nabla m(x_0)|\big\}\Big\}.
\end{align*}
Here, when $x\in\Sigma_{1}\cup\Sigma_{2}$,
$\kappa_1(x), \kappa_2(x),\cdots,\kappa_N$ are the eigenvalues of
$D^2m(x)$; when $x\in\Sigma_{3}$, $\kappa_{N}=0$ and
$\kappa_1(x), \kappa_2(x),\cdots,\kappa_{N-1}(x)$ are the eigenvalues of
$D^2m_{\partial\Omega}(x)$.
\end{thm}

 It is easily seen that $\Sigma_1\cup\Sigma_2\cup\Sigma_3=\Sigma_1^*\cup\Sigma_2^*\not=\emptyset$. In order to obtain Theorems \ref{theorem1} and \ref{th1.2}, our approach mainly follows that of \cite{CL2}, which heavily relies on the variational structure of the problems under consideration; nevertheless, some nontrivial ingredients are introduced here to overcome the difficulties caused by the boundary conditions.  Notice that the eigenvalue problem \eqref{eq:1-D} has no variational structure in general; some very different approaches were used in \cite{DF,Fr} to derive Theorem \ref{theorem1-D}.

There are close connections and delicate differences among Theorems \ref{theorem-N}, \ref{theorem1} and \ref{th1.2}. In particular, we would like to make the following comments.

\begin{remark}\label{r-D2}

\begin{itemize}
\item[\rm{(i)}] Theorem \ref{th1.2} reduces to Theorem \ref{theorem-N} if $k=0$.

\item[\rm{(ii)}]  Different from Theorem \ref{theorem1}(i) for the Dirichlet boundary problem \eqref{eq:1}, $\lim_{D\to 0}\lambda(D)$ is
finite for the Robin boundary problem \eqref{eq:3.1} even if $\Sigma_{1}\cup\Sigma_{2}=\emptyset$;
however, $\lim_{D\to 0}\lambda(D)\to+\infty$ as $k\to\infty$ in Theorem \ref{th1.2} provided that
$\Sigma_{1}\cup\Sigma_{2}=\emptyset$ and $\beta>0$ on $\partial\Omega$.
On the other hand, if $\beta>0$ on $\partial\Omega$ and $k$ is properly large in problem \eqref{eq:3.1}, the limit formula in Theorem \ref{th1.2} coincides with the one in Theorem \ref{theorem1}(ii). Hence, Theorem \ref{theorem1} can be regarded as the limiting behavior of Theorem \ref{th1.2} as $k\to\infty$.

\item[\rm{(iii)}] One should observe that for the Dirichlet boundary problem \eqref{eq:1}, the set $\Sigma_3$ does not affect the limit $\lim_{D\to 0}\lambda(D)$ (when it is finite); this is very different from the behavior of the principal eigenvalue for the Neumann boundary problem and the Robin boundary problem \eqref{eq:3.1} as stated by Theorems \ref{theorem-N} and \ref{th1.2}.

\item[\rm{(iv)}] As in Remark \ref{r-N}, when $N=1$ and $\Omega$ is a bounded open interval,
Theorem \ref{th1.2} remains valid without the assumption
that $\mathrm{det}(D^{2}m(x))\neq 0$ and $n(x)$ is an eigenvector of
  $D^{2}m(x)$ for all $x\in\Sigma_{2}$.

\end{itemize}
\end{remark}

Finally, we study the asymptotic behavior of the principal eigenvalue as $D\to +\infty$ of the eigenvalue problem
\begin{equation}
  \label{eq:4.1}
\begin{cases}
  -D\Delta \phi -2\alpha\nabla m(x)\cdot \nabla\phi+V(x)\phi=\lambda\phi &\hbox{ in }\Omega,\\
  \frac{\partial\phi}{\partial n}+\beta(x)\phi=0 &\hbox{ on }\partial\Omega,
\end{cases}
\end{equation}
where $\beta\in C(\partial\Omega)$ is a given function.

If $\beta=0$, it is well known that
 $$
 \lim_{D\to+\infty}\lambda(D)=\frac{1}{|\Omega|}\int_\Omega V.
 $$
We aim to explore the asymptotic behavior of the principal eigenvalue as $D\to +\infty$ in the general setting above.

Let $\mu_1$ be the principal eigenvalue of the following eigenvalue problem
\begin{equation}
  \label{eq:4.2}
\begin{cases}
  -\Delta \phi=\mu\phi &\hbox{ in }\Omega,\\
  \frac{\partial\phi}{\partial n}+\beta(x)\phi=0 &\hbox{ on
  }\partial\Omega.
\end{cases}
\end{equation}

We now state the last main result of this paper.
\begin{thm}\label{th1.3} Let $\lambda(D)$ be the principal eigenvalue of \eqref{eq:4.1}. The following assertions hold.
\begin{itemize}
\item[\rm{(i)}] If $\mu_1>0$, then $\lim_{D\to
    +\infty}\lambda(D)=+\infty$.

\item[\rm{(ii)}] If $\mu_1<0$, then $\lim_{D\to
    +\infty}\lambda(D)=-\infty$.

\item[\rm{(iii)}] If $\mu_1=0$, then
  \[
    \lim_{D\to
      +\infty}\lambda(D)=\int_{\Omega}\left(V\phi_{0}^{2}-2\alpha\nabla
    m\cdot\nabla\phi_{0}\right),
  \]
  where $\phi_{0}$ is the principal eigenfunction of
  \eqref{eq:4.2} corresponding to $\mu_1=0$ such that $\int_{\Omega}\phi_{0}^{2}=1$.

\end{itemize}

In particular, {\rm (i)} holds if $\beta\geq,\not\equiv0$, {\rm (ii)} holds if either $\int_{\partial\Omega}\beta<0$ or $\int_{\partial\Omega}\beta=0$ and $\beta\not\equiv0$, and {\rm (iii)} holds if $\beta\equiv0$.

\end{thm}

One may further refer to \cite{BHN,CL1,Ec,PZ1,PZ2,PZh,Ve1,Ve2,W} and the references therein for related research works
on the eigenvalue problems considered in the current paper.

The rest of our paper is organized as follows. In section 2, we investigate the asymptotic behavior of the principal eigenvalue of
problems \eqref{eq:1} and \eqref{eq:3.1} as $D\to0$, and prove Theorems \ref{theorem1} and \ref{th1.2}. Section 3 concerns the
asymptotic behavior of the principal eigenvalue of problem \eqref{eq:4.1} as $D\to\infty$ and Theorem \ref{th1.3} is established; in one space dimension, some improved results are obtained. In section 4, as an application of our theoretical results, we study a reaction-diffusion-advection equation which is used to describe the evolution of a single species living in a heterogeneous stream environment, and find some interesting effects of buffer zone and small/large diffusion rate on the species persistence and extinction.

%%%%%%%%%%%%%%%%%%%%%%%%%%%%%%%%%%%%%%%%%%%%%%%%%%%%%%%%%%%%%%%%%%%%%%%%%%%%%%%%%%%%%%%%%%%%%%%%%%%%%%%%%%%%%%%%%%%%%%%
\section{Asymptotic behavior as $D\to 0$:\ Proof of Theorems \ref{theorem1} and \ref{th1.2}}
\subsection{Dirichlet boundary problem \eqref{eq:1}}
In this section we consider the eigenvalue problem \eqref{eq:1}. Equation (\ref{eq:1}) can be rewritten in
the divergence form
\begin{equation}
  \label{eq:2}
\begin{cases}
  -D\nabla\cdot[e^{2(\alpha/D)m}\nabla\phi]+e^{2(\alpha/D)m}V\phi=\lambda
  e^{2(\alpha/D)m}\phi &\hbox{ in }\Omega,\\
  \phi=0 &\hbox{ on }\partial\Omega.
\end{cases}
\end{equation}

It is known that the principal eigenvalue $\lambda(D)$ can be characterized by
\begin{equation}\label{eq:3}
  \begin{aligned}
    \lambda(D)&=\inf_{\phi\in W_0^{1,2}(\Omega),\phi\not\equiv
      0}\frac{\int_{\Omega}e^{2\alpha
        m/D}(D|\nabla\phi|^2+V\phi^2)}{\int_{\Omega}e^{2\alpha
        m/D}\phi^2}\\
    &=\inf_{w\in W_{0}^{1,2}(\Omega),
      \int_{\Omega}w^{2}=1}\int_{\Omega}D\big|\nabla
    w-\frac{\alpha}{D}w\nabla m\big|^{2}+Vw^{2}.
  \end{aligned}
\end{equation}
Indeed, the second variational characterization in \eqref{eq:3} is derived through
the substitution $\phi=e^{-\alpha m/D}w$ in \eqref{eq:1}. Clearly, $w=e^{\alpha m/D}\phi$
solves
 \begin{equation}
  \label{eq:2-1}
\begin{cases}
  -D\Delta w+\left(\frac{\alpha^2}{D}|\nabla m|^2+\alpha\Delta m+V\right)w=\lambda
  (D)w &\hbox{ in }\Omega,\\
  w=0 &\hbox{ on }\partial\Omega.
\end{cases}
\end{equation}

Let us define the following functional
\[
  E(u)=\int_{\Omega}D\left|\nabla u-\frac{\alpha}{D}u\nabla
    m\right|^{2}=\int_{\Omega}D\left|\nabla\ln
    u-\frac{\alpha}{D}\nabla m\right|^{2}u^{2},\ \ \forall u\in H_{0}^{1}(\Omega).
\]
Then we have
\begin{lemma}\label{lemma1} For any $u\in H_{0}^{1}(\Omega)$, there holds
\[
  E(u)\geq \frac{\alpha^{2}}{D}\int_{\Omega}\left[|\nabla
    m|^{2}+\frac{D}{\alpha}\Delta m\right]u^{2}.
\]
\end{lemma}
\begin{proof} Basic calculation yields
\begin{align*}
  E(u)&=\int_{\Omega}D|\nabla u|^{2}-2\alpha u\nabla u\cdot\nabla
        m+\frac{\alpha^{2}}{D}u^{2}|\nabla m|^{2}\\
      & =\int_{\Omega}D|\nabla u|^{2}+\alpha u^{2}\Delta
        m+\frac{\alpha^{2}}{D}u^{2}|\nabla m|^{2}\\
      &\geq \frac{\alpha^{2}}{D}\int_{\Omega}\left[|\nabla m|^{2}+\frac{D}{\alpha}\Delta
        m\right]u^{2}.
\end{align*}
\end{proof}

Lemma \ref{lemma1} indicates that when $D\to 0$, the mass of $w^2$ for the principal eigenfunction $w$
is mostly concentrated on the critical points of $m$. Thus, it is natural to investigate the behavior of the principal eigenfunction $w$ locally near the critical points of $m$.

Let $\zeta$ be a smooth function. Multiplying the differential equation
(\ref{eq:2-1}) by $\zeta w$ and integrating over $\Omega$, we obtain
\begin{align*}
  \int_{\Omega}\lambda(D)\zeta^{2}w^{2}
  &=\int_{\Omega}-D\zeta^{2}w\Delta
    w+\frac{\alpha^{2}}{D}|\nabla
    m|^{2}\zeta^{2}w^{2}+\alpha\zeta^{2}w^{2}\Delta
    m+V\zeta^{2}w^{2}\\
  &=\int_{\Omega}D\nabla w\cdot\nabla(\zeta^{2}w)+\frac{\alpha^{2}}{D}|\nabla
    m|^{2}\zeta^{2}w^{2}+\alpha\zeta^{2}w^{2}\Delta
    m+V\zeta^{2}w^{2}.
\end{align*}

By setting $W=\zeta w$, it then follows that
\begin{align*}
  \int_{\Omega}\lambda(D)\zeta^{2}w^{2}
  &=\int_{\Omega}D\nabla
    w\cdot(\zeta\nabla(\zeta
    w)+\zeta w\nabla\zeta)+\frac{\alpha^{2}}{D}|\nabla
    m|^{2}\zeta^{2}w^{2}-2\alpha\zeta
    w\nabla(\zeta w)\cdot\nabla
    m+V\zeta^{2}w^{2}\\
  &=\int_{\Omega}D\left|\nabla(\zeta w)-\frac{\alpha}{D}\zeta
    w\nabla
    m\right|^{2}+V\zeta^{2}w^{2}-Dw\nabla\zeta\cdot\nabla(\zeta
    w)+D\zeta w\nabla w\cdot\nabla\zeta \\
  &=\int_{\Omega}D\left|\nabla W-\frac{\alpha}{D}W\nabla m\right|^{2}+VW^{2}-D\int_{\Omega}w^{2}|\nabla\zeta|^{2}.
\end{align*}
Thus we have
\begin{lemma}\label{lemma2}  Let $(\lambda(D),w)$ be the solution of \eqref{eq:2-1} and $\zeta$ be a
  smooth function. Then $W:=\zeta w$ satisfies
\[
  \int_{\Omega}\left(D\left|\nabla W-\frac{\alpha}{D} W\nabla
      m\right|^{2}+VW^{2}\right)=\lambda(D)\int_{\Omega}W^{2}+D\int_{\Omega}w^{2}|\nabla\zeta|^{2}.
\]
\end{lemma}

We can further establish the following estimates.

\begin{lemma}\label{lemma3} The following assertions hold.

\begin{itemize}
\item[\rm{(i)}] Assume that $B(x_{0},R)\subset\Omega$, $W=0$ in
  $\Omega\setminus B(x_{0},R)$. Let
  $\kappa_{1}(x_{0}),\cdots,\kappa_{N}(x_{0})$ be the eigenvalues of
  $D^{2}m(x_{0})$. Then for $C=2N\|D^{3}m\|_{L^{\infty}(\Omega)}$,
  \begin{equation}\label{eq:4}
    E(W)\geq\alpha\left[\sum_{i=1}^{N}\left(|\kappa_{i}(x_{0})|+\kappa_{i}(x_{0})\right)-CR\right]\int_{\Omega}W^{2}.
  \end{equation}

\item[\rm{(ii)}] Assume that $x_{0}\in\partial\Omega$, $W=0$ in
  $\Omega\setminus B(x_{0},R)$ and $W=0$ on
  $\partial\Omega\cap B(x_{0},R)$. Let
  $\kappa_{1}(x_{0}),\cdots,\kappa_{N}(x_{0})$ be the eigenvalues of
  $D^{2}m(x_{0})$. Then for $C=2N\|D^{3}m\|_{L^{\infty}(\Omega)}$,
  \begin{equation}\label{eq:4}
    E(W)\geq\alpha\left[\sum_{i=1}^{N}\left(|\kappa_{i}(x_{0})|+\kappa_{i}(x_{0})\right)-CR\right]\int_{\Omega}W^{2}.
  \end{equation}

\end{itemize}

\end{lemma}

\begin{proof}
    Let $\{e_{1}(x_{0}),\cdots,e_{N}(x_{0})\}$ be an orthonormal
    eigenbasis of $D^{2}m(x_{0})$ with the corresponding eigenvalues
    $\kappa_{1}(x_{0}),\cdots,\kappa_{N}(x_{0})$. For sake of simplicity, we
    abbreviate $e_{i}(x_{0})$ and $\kappa_{i}(x_{0})$ as $e_{i}$ and
    $\kappa_{i}$ respectively. Then we deduce
    \begin{align*}
      E(W)&=\sum_{i=1}^{N}\int_{\Omega}D\left|e_{i}\cdot\left(\nabla
            W-\frac{\alpha W}{D}\nabla m\right)\right|^{2}\\
          &=\sum_{i=1}^{N}\int_{\Omega}D\left|e_{i}\cdot\nabla
            W+\mathrm{sgn}(\kappa_{i})\frac{\alpha W}{D}\nabla
            m\right|^{2}+J\\
          &\geq J,
    \end{align*}
    where $\mathrm{sgn}(s)=1$ if $s>0$ and $\mathrm{sgn}(s)=-1$ if $s\leq0$
    and
    \[
      J:=-\sum_{\kappa_{i}>0}\int_{\Omega}2\alpha(e_{i}\cdot\nabla
      W^{2})(e_{i}\cdot\nabla m).
    \]
    Since $e_{i}D^{2}m(x_{0})=\kappa_{i}e_{i}$, we obtain
    \begin{align*}
      J&=2\alpha\sum_{\kappa_{i}>0}\int_{\Omega}-\mathrm{div}(W^{2}(e_{i}\cdot\nabla
         m)e_{i})+W^{2}e_{i}D^{2}m e_{i}^{T}\\
       &=\alpha\sum_{\kappa_{i}>0}\int_{\Omega}\left[2\kappa_{i}+2e_{i}(D^{2}m(x)-D^{2}m(x_{0}))e_{i}^{T}\right]W^{2}\\
       &\geq
         \alpha\left[\sum_{i=1}^{N}(|\kappa_{i}|+\kappa_{i})-2N\|D^{3}m\|_{L^{\infty}(\Omega)}R\right]\int_{\Omega}W^{2}.
    \end{align*}
    Thus the assertion (i) holds and (ii) follows similarly.
  \end{proof}

With the aid of the previous lemmas, we are now ready to present
 \vskip6pt
\noindent
 {\bf Proof of Theorem \ref{theorem1}:} We first verify the assertion (i). By our assumption, there exists a constant $\delta>0$ such that $|\nabla m(x)|\geq
\delta$ for all $x\in\bar{\Omega}$. Lemma \ref{lemma1} implies,  for every
$w\in H_{0}^{1}(\Omega)$,
\[
E(w)\geq \left(\frac{\alpha^{2}\delta}{D}-\alpha|\Delta m|_{L^{\infty}(\Omega)}\right)\int_{\Omega}w^{2}.
\]
Thus by \eqref{eq:3}, it holds
\[
  \lambda(D)\geq \frac{\alpha^{2}\delta}{D}-\alpha|\Delta
  m|_{L^{\infty}(\Omega)}+\min_{x\in\bar{\Omega}}V(x)\to +\infty,
\]
as $D\to 0$.\\

We next show the assertion (ii).  The proof is similar that of \cite[Theorem 8.1]{CL2};
however necessary modifications are needed.  We first estimate the lower bound
\[
  \liminf_{D\to 0}\lambda(D)\geq\min_{x\in\Sigma_{1}\cup\Sigma_{2}}\Big\{V(x)+\alpha\sum_{i=1}^N(|\kappa_i(x)+\kappa_i(x))\Big\}.
\]

Let $w$ be the principal eigenfunction normalized by $\int_{\Omega}w^{2}=1$
and $R$ be a small number. We first cover $\Sigma_{2}$ by balls
$\{B(x_{k},R/3)\}_{k=1}^{K_{1}}$ such that $x_{k}\in\Sigma_{2}$ and
$|x_{k}-x_{l}|\geq R/3$ for all $1\leq k<j\leq K_{1}$. Let
$$R_{1}=\frac{1}{2}\min\Big\{R,\mathrm{dist}(\partial\Omega,\Sigma_{1}\setminus(\bigcup_{k=1}^{K_{1}}B(x_{k},R/2)))\Big\}.$$ Then
$R_{1}>0$ and we cover
$\Sigma_{1}\setminus(\bigcup_{k=1}^{K_{1}}B(x_{k},R/2)))\}$ by balls
$\{B(x_{k},R_{1}/3)\}_{k=K_{1}+1}^{K_{2}}$ such that
$x_{k}\in \{B(x_{k},R_{1}/3)\}_{k=K_{1}+1}^{K_{2}}$ for
$k=K_{1}+1,\cdots,K_{2}$ and $|x_{k}-x_{l}|\geq R_{1}/3$ for all
$K_{1}+1\leq k<l \leq K_{2}$.

Denote
$$\Omega_{0}=\mathbb{R}^{N}\setminus\Big(\bigcup_{k=1}^{K_1}\overline{B(x_{k},R/2)}\cup
  \bigcup_{k=K_{1}+1}^{K_{2}}\overline{B(x_{k},R_{1}/2)}\Big)$$ and let
$\{\zeta_{k}^{2}\}$ be a partition of unit subordinated to open
covering $\{B(x_{1},R),\cdots,B(x_{K_{2}},R),\Omega_{0}\}$ with
\[
\sum_{k=0}^{K_2}\zeta_{k}^{2}(x)=1,\ \ \forall x\in\mathbb{R}^{N},
\]
\[
\zeta_{k}=0 \hbox{ in } \mathbb{R}^{N}\setminus B(x_{k},R),\ \
|\nabla\zeta_{k}|\leq \frac{C}{R} \hbox{ in }\mathbb{R}^{N},\ \forall k=1,\cdots,K_{1},
\]
\[
\zeta_{k}=0 \hbox{ in } \mathbb{R}^{N}\setminus B(x_{k},R),\ \
|\nabla\zeta_{k}|\leq \frac{C}{R_{1}} \hbox{ in }\mathbb{R}^{N},\ \forall k=K_{1}+1,\cdots,K_{2},
\]
and $\zeta_{0}=0$ in $\Omega_{0}$. For each
$x\in\mathbb{R}^{N}$, there exists at most $4^{N}$ number of indexes
$k\geq 1$ such that $\zeta_{k}\neq 0$. As a result, we have
\[
\sum_{k=1}^{K_{2}}|\nabla\zeta_{k}(x)|^{2}\leq \frac{C(N)}{R_{1}^{2}},\ \ \
\forall x\in\mathbb{R}^{N}.
\]
By setting $W_{k}=\zeta_{k}w$ for $k=0,1,\cdots,K_{2}$, we get from Lemma
\ref{lemma3} that for $1\leq i\leq N$,
\[
E(W_{k})\geq \alpha\left[\sum_{i=1}^{N}(|\kappa_{i}(x_{k})|+\kappa_{i}(x_{k}))-CR\right]\int_{\Omega}W_{k}^{2}.
\]
In light of Lemma \ref{lemma2} we further obtain
\begin{align*}
  &\sum_{k=1}^{K_{2}}\lambda(D)\int_{\Omega}W_{k}^{2}+D\sum_{k=1}^{K_{2}}\int_{\Omega}|\nabla\zeta_{k}|^{2}w\\
  &=\sum_{k=1}^{K_{2}}\int_{\Omega}D\left|\nabla
  W_{k}-\frac{\alpha W_{k}}{D}\nabla m\right|^{2}+VW_{k}^{2}\\
  &\geq\Big(\min_{1\leq k\leq K_2}\big\{\sum_{i=1}^{N}(|\kappa_{i}(x_{k})|+\kappa_{i}(x_{k}))+V(x_{k})\big\}-CR\Big)\sum_{k=1}^{K_{2}}\int_{\Omega}W_{k}^{2}.
\end{align*}
Due to $x_{k}\in\Sigma_{1}\cup\Sigma_{2}$ for all
$k=1,2,\cdots,K_2$ and $|\nabla\zeta_{k}|\leq C/R_{1}$ for all
$i=1,2,\cdots,K_2$, it follows that
\begin{equation}
  \label{eq:5}
  \Bigg(\lambda(D)-\min_{x\in\Sigma_{1}\cup\Sigma_{2}}\Big\{\sum_{i=1}^{N}(|\kappa_{i}(x)|
  +\kappa_{i}(x))+V(x)\Big\}\Bigg)\int_{\Omega}w^{2}\sum_{k=1}^{K_2}\zeta_{k}^{2}\geq -CR-\frac{C(N)D}{R_{1}^{2}}.
\end{equation}

On the other hand, there exists a positive constant $\delta$ such that
\[
  |\nabla m(x)|>\delta \hbox{ for all
  }x\in\Omega\setminus\Big(\bigcup_{k=1}^{K_1}B(x_{k},R/2)\cup
    \bigcup_{k=K_{1}+1}^{K_{2}}B(x_{k},R_{1}/2)\Big).
\]
Making use of Lemma \ref{lemma1}, one infers
\[
  \int_{\Omega}W_0^{2}\leq
  \int_{\Omega\setminus\cup_{k=1}^{N}B(x_{k},R/2)}w^{2}\leq
  \frac{DC(R)(1+\alpha)}{\alpha^{2}},
\]
and hence
\[
  \int_{\Omega}w^{2}\sum_{k=1}^{K_2}\zeta_{k}^{2}=1-\int_{\Omega}W_0^{2}\geq
  1-\frac{DC(R)(1+\alpha)}{\alpha^{2}}.
\]

Therefore, letting $D\to0$, we see from (\ref{eq:5}) that
\[
  \liminf_{D\to 0}\lambda(D)\geq
  \min_{x\in\Sigma_{1}\cup\Sigma_{2}}\Big\{\sum_{i=1}^{N}(|\kappa_{i}(x)|+\kappa_{i}(x))+V(x)\Big\}-CR,
\]
from which we have (by sending $R\to 0$) that
\[
  \liminf_{D\to 0}\lambda(D)\geq
  \min_{x\in\Sigma_{1}\cup\Sigma_{2}}\Big\{V(x)+\alpha\sum_{i=1}^N(|\kappa_i(x)+\kappa_i(x))\Big\}.
\]

In the sequel, we are going to show
\[
  \limsup_{D\to 0}\lambda(D)\leq
  \min_{x\in\Sigma_{1}\cup\Sigma_{2}}\Big\{V(x)+\alpha\sum_{i=1}^N(|\kappa_i(x)+\kappa_i(x))\Big\}.
\]
As it can be seen below, there are two cases to handle.

{\it Case 1. Let $x_{0}\in\Sigma_{1}$}. Denote by $\kappa_{1},\kappa_{2},\cdots,\kappa_{N}$ the eigenvalues of
$D^{2}m(x_{0})$. By translation and rotation we may assume that
$x_{0}=0$, $D^{2}m(x_{0})=\mathrm{diag}(\kappa_{1},\cdots,\kappa_{N})$
and $m_{x_{i}}(x)=\kappa_{i} x_{i}+O(|x|^{2})$. Fix an arbitrarily small
positive constant $\delta$ and an arbitrarily large positive constant
$M$, let us define
 $$c_{i}=e^{-\frac{1}{2}(|\kappa_{i}|+\delta)M^{2}},\ \
q_{i}(x_{i})=e^{-\frac{1}{2}(|\kappa_{i}|+\delta)x_{i}^{2}},$$
and
\[
  p_{i}(x_{i})=\Big[e^{-\frac{1}{2}(|\kappa_{i}|+\delta)x_{i}^{2}}-e^{-\frac{1}{2}(|\kappa_{i}|+\delta)M^{2}}\Big]^{+},\ \
  p(x)=\prod_{i=1}^{N}p_{i}(x_{i}).
\]
Take
$\epsilon=\sqrt{\frac{D}{\alpha}}$ and
  \begin{align*}
    \zeta(x)&=\frac{1}{\epsilon^{N/2}}p(\frac{x}{\epsilon})\\
            &=\frac{1}{\epsilon^{N/2}}\prod_{i=1}^{N}\Big[e^{-\frac{1}{2}(|\kappa_{i}|+\delta)\frac{x_{i}^2}{\epsilon^{2}}}-e^{-\frac{1}{2}(|\kappa_{i}|+\delta)M^{2}}\Big]^{+}.
  \end{align*}

We choose a small constant $r_{0}>0$ such that
$B(0,r_{0})\subset \Omega$. If $\frac{r_{0}\delta}{\epsilon}>M$, then
$C(\epsilon,\delta,
M)=\int_{\Omega}\zeta^{2}=\int_{\mathbb{R}^{N}}p^{2}=:C(\delta,
M)$ and
\[
\lim_{M\to \infty}C(\delta,M)=\prod_{i=1}^{N}\sqrt{\frac{\pi}{|\kappa_{i}|+\delta}}=:C(\delta).
\]
Let $\epsilon$ be sufficiently small and
$w(x)=\frac{\zeta(x)}{C(\delta,M)}$. It follows from (\ref{eq:3}) that
\begin{align*}
  \lambda(D)&\leq \int_{\mathbb{R}^{N}}\Big[D\big|\nabla(\ln
                     w)-\frac{\alpha}{D}\nabla m\big|^{2}+V\Big]w^{2}\\
                   &=\int_{\mathbb{R}^{N}}\frac{\alpha}{\epsilon^{2}}\Bigg\{\sum_{i=1}^{N}
                   \Big[\frac{q_{i}(\epsilon^{-1}x_{i})}{p_{i}(\epsilon^{-1}x_{i})}(|\kappa_{i}|+\delta)x_{i}
                   +\kappa_{i}x_{i}+O(|x|^{2})\Big]^{2}+V\Bigg\}w^{2}\\
                   &=\int_{\mathbb{R}^{N}}\Bigg\{\alpha\Big[\sum_{i=1}^{N}(|\kappa_{i}|+\kappa_{i}+\delta)^{2}y_{i}^{2}+\sum_{i=1}^{N}\frac{c_{i}^{2}}{p_{i}(y_{i})^{2}}(|\kappa_{i}|+\delta)\\
                   &\ \ \ +\frac{2c_{i}(|\kappa_{i}|+\kappa_{i}+\delta)}{p_{i}(y_{i})}+O(\epsilon)\Big]+V(x_{0}+\epsilon y)\Bigg\}\frac{p(y)^{2}}{C(\delta,M)}dy.
\end{align*}
One can easily check that, as $M\to +\infty$,
\[
  \int_{\mathbb{R}^{N}}\frac{c_{i}^{2}}{p_{i}(y_{i})^{2}}p(y)^{2}dy\to
  0,\ \ \
\int_{\mathbb{R}^{N}}\frac{c_{i}}{p_{i}(y_{i})}p(y)^{2}dy \to 0.
\]
Thus, sending $M\to +\infty$ and $\epsilon\to 0$ gives
\begin{align*}
  \limsup_{D\to 0}\lambda(D)&\leq\frac{1}{C(\delta)}\int_{\mathbb{R}^{N}}\exp\Big(-\sum_{i=1}^{N}(|\kappa_{i}|+\delta)y_{i}^{2}\Big)
  \Big[\alpha\sum_{i=1}^{N}(|\kappa_{i}|+\kappa_{i}+\delta)y_{i}^{2}+V(x_{0})\Big]dy\\
                                    &=V(x_{0})+\frac{\alpha}{2}\sum_{i=1}^{N}\frac{(|\kappa_{i}|+\kappa_{i}+\delta)^{2}}{|\kappa_{i}|+\delta}.
\end{align*}
Finally, sending $\delta\to 0$, we have
\[
\limsup_{D\to 0}\lambda(D)\leq V(x_{0})+\alpha\sum_{i=1}^{N}(|\kappa_{i}|+\kappa_{i}).
\]

{\it Case 2. Let $x_{0}\in\Sigma_{2}$.} By translation and
rotation we may assume that $x_{0}=0$,
$n(x_{0})=(0,0,\cdots,0,-1)$ and
$D^{2}m(x_{0})=\mathrm{diag}(\kappa_{1},\cdots,\kappa_{N})$ with $\kappa_{N}=0$. Given  $M>0$, we define
$$p(x)=\prod_{i=1}^{N}p_{i}(x_{i}),$$ where
\[
p_{i}(x_{i})=\Big(e^{-\frac{1}{2}(|\kappa_{i}|+\delta)x_{i}^{2}}-e^{-\frac{1}{2}(|\kappa_{i}|+\delta)M^{2}}\Big)^{+}
\]
for $1\leq i\leq N-1$, and
\[
  p_{N}(x_{N})=\Big(e^{-\frac{1}{2}\delta (x_{N}-M-1)^{2}}-e^{-\frac{1}{2}\delta M^{2}}\Big)^{+}.
\]

Let $\epsilon=\sqrt{\frac{D}{\alpha}}$ and
$\zeta(x)=\frac{1}{\epsilon^{N/2}}p(\frac{x}{\epsilon})$. There exists
$\epsilon(M)>0$ such that for $\epsilon<\epsilon(M)$,
$\mathrm{supp}\zeta\subset\Omega$,
$\int_{\Omega}\zeta^{2}=\int_{\mathbb{R^{N}}}p^{2}=:C(\delta,M)$,
and
\[
\lim_{\tau\to 0,M\to \infty}C(\delta,M)=\frac{1}{2}\prod_{i=1}^{N}\sqrt{\frac{\pi}{|\kappa_{i}|+\delta}}=C(\delta).
\]

Let $w(x)=\frac{\zeta(x)}{C(\delta,M)}$, then
$\int_{\Omega}w^{2}(x)=1$. The analysis similar to Case 1 shows that for $1\leq i\leq N-1$,
\[
  \int_{\mathbb{R}^{N}}D\big|\partial_{i}\ln
  w-\frac{\alpha}{D}\partial_{i}m\big|^{2}w^2\to
  \frac{\alpha(|\kappa_{i}|+\kappa_{i}+\delta)}{2(|\kappa_{i}|+\delta)}
\]
as $M\to +\infty$, $\epsilon\to 0$, and when $i=N$,
\begin{align*}
  &\int_{\mathbb{R}^{n}}D\big|\partial_{N}\ln
    w-\frac{\alpha}{D}\partial_{N}m\big|^{2}w^{2}\\
  &=\int_{\mathbb{R}^{N}}\frac{\alpha}{\epsilon^{2}}\left[\epsilon\exp{\left(-\frac{1}{2}\delta(\frac{x_{N}}{\epsilon}-M-1)^{2}\right)}
  p_{N}^{-1}(\frac{x_{N}}{\epsilon})\delta(\frac{x_{N}}{\epsilon}-M-1)+O(|x|^{2})\right]^{2}w^{2}\\
  &=\int_{\mathbb{R}^{N}}\alpha\left[\exp{\left(-\frac{1}{2}\delta(y_{N}-M-1)^{2}\right)}p_{N}^{-1}(y_{N})\delta(y_{N}-M-1)+O(\epsilon
    M^{2})\right]^{2}w^{2}\\
  &=\int_{\mathbb{R}^{N}}\alpha\Big[\exp{\left(-\delta(y_{N}-M-1)^{2}\right)}p_{N}^{-2}(y)\delta^{2}(y_{N}-M-1)^{2}+O(\epsilon
    M^{3})\Big]w^{2}.
\end{align*}
We first choose a sequence $M_{k}\to \infty$ and then take $\epsilon_{k}$
with $\epsilon_{k}=o(M_{k}^{-3})$. Passing to the limit, we have
\[
  \int_{\mathbb{R}^{n}}D\big|\partial_{N}\ln
  w-\frac{\alpha}{D}\partial_{N}m\big|^{2}w^{2} \to
  \frac{\alpha}{2}\delta.
\]
Letting $\delta\to 0$, it follows from (\ref{eq:3}) that
\[
  \limsup_{D\to 0}\lambda(D)\leq
  V(x_{0})+\alpha\sum_{i=1}^{N}(|\kappa_{i}|+\kappa_{i}).
\]
Therefore, the desired estimate is verified. The proof is now complete. {\hfill $\Box$}

%%%%%%%%%%%%%%%%%%%%%%%%%%%%%%%%%%%%%%%%%%%%%%%%%%%%%%%%%%%%%%%%%%%%%%%%%%%%%%%%%%%%%%%%%%%%%%%%%%%%%%%%%%%%%%%%%%%%%%%%%%%%%%%%%%%%%%%%%%%%%%%%%%%%%%%%%%%%%%%%%%%%%%%%%%%
\subsection{Robin boundary problem \eqref{eq:3.1}}
In this section we consider the eigenvalue problem \eqref{eq:3.1}.
The principal eigenvalue possesses the following variational characterization:
\begin{equation}\label{eq:3.2}
  \begin{aligned}
    \lambda(D)&=\inf_{\phi\in W^{1,2}(\Omega),\phi\not\equiv
      0}\frac{\int_{\Omega}e^{2\alpha
        m/D}(D|\nabla\phi|^2+V\phi^2)+Dk\int_{\partial\Omega}e^{2\alpha
        m/D}\beta\phi^2}{\int_{\Omega}e^{2\alpha
        m/D}\phi^2}\\
    &=\inf_{w\in W^{1,2}(\Omega),
      \int_{\Omega}w^{2}=1}\int_{\Omega}D|\nabla w-\frac{\alpha}{D}w\nabla m|^{2}+Vw^{2}+Dk\int_{\partial\Omega}\beta w^2.
  \end{aligned}
\end{equation}
Furthermore, we may assume that $w=e^{\frac{\alpha m}{D}}\phi$ satisfies
\begin{equation}
  \label{eq:3.3}
\begin{cases}
  -D\Delta w+\left(\frac{\alpha^2}{D}|\nabla m|^2+\alpha\Delta m+V-\mu\right)w=0 \hbox{ in }\Omega,\\
  \frac{\partial w}{\partial n}-\frac{\alpha}{D}w\frac{\partial
    m}{\partial n}+k\beta w=0 \hbox{ on
  }\partial\Omega,\int_{\Omega}w^2=1.
\end{cases}
\end{equation}

Before going further, let us recall some notations introduced in \cite[Section
3]{CL2}. Let $d(x)$ be the signed distance from
$x\in\mathbb{R}^{N}$ to $\partial\Omega$ which is positive if
$x\in\Omega$ and negative if $x\not\in\bar{\Omega}$. Since $\Omega$ is
smooth, there exists a constant $R_{0}$ such that $d(x)$ is smooth in
the $R_{0}$-neighborhood of $\partial\Omega$: $\partial\Omega(R_0)=\{x\in\mathbb{R}^{N}:\ {\mbox{dist}}(x,\partial\Omega)<R_0\}$. Note that
$\nabla d(x)=-n(x)$ for all $x\in\partial\Omega$. We may extend $n(x)$
to $R_{0}$ neighborhood of $\partial\Omega$ by $n(x)=-\nabla d(x)$. We
denote by $m_{\partial\Omega}$ the restriction of $m$ to the boundary
of $\Omega$. We can extend the definition of $m_{\partial\Omega}$ to
$\partial \Omega(R_{0})$ by
\[
  m_{\partial\Omega}(x):=m(x+d(x)n(x)),\ \ \forall
  x\in\partial\Omega(R_{0}).
\]
Restrict $D^{2}m(x+d(x)n(x))$ to the tangent space of
$\partial\Omega$, we define
\[
  D^{2}m_{\partial\Omega}(x):=[I-n\otimes n]D^{2}m(x)[I-n\otimes
  n]-(n\cdot \nabla m)\nabla^{T}n,
\]
where $n\otimes n=n^T n$ and $\nabla^{T}n=-\nabla^T\nabla d=-D^2d$.
Then $\{0, n(x)\}$ is an eigenpair of $D^{2}m_{\partial\Omega}(x)$. We
set $\kappa_{N}=0$ and denote by $\kappa_{1},\cdots,\kappa_{N-1}$ the
eigenvalues of $D^{2}m_{\partial\Omega}(x)$ in the tangent space of
$\partial\Omega$. We call $\kappa_{1},\cdots,\kappa_{N-1}$ the eigenvalues of the Hessian of the
restriction of $m$ of $\partial\Omega$.\\

In what follows, we define
\begin{align*}
  F(w)&=\int_{\Omega}D\left|\nabla w-\frac{\alpha}{D}w\nabla
        m\right|^{2}+Dk\int_{\partial\Omega}\beta w^2\\
      &=\int_{\Omega}D\left|\nabla\ln
        w-\frac{\alpha}{D}\nabla m\right|^{2}w^{2}+Dk\int_{\partial\Omega}\beta w^2.
\end{align*}
It is easily seen that
  \begin{equation}\label{eq:3.4}
  \begin{aligned}
    F(w)&=\int_{\Omega}D|\nabla w|^{2}-2\alpha w\nabla
    w\cdot\nabla m+\frac{\alpha^{2}w^{2}}{D}|\nabla m|^{2}+Dk\int_{\partial\Omega}\beta w^2\\
    &=\int_{\Omega}D|\nabla w|^{2}+\alpha w^{2}\Delta
    m+\frac{\alpha^{2}w^{2}}{D}|\nabla
    m|^{2}-\int_{\partial\Omega}\alpha w^{2}\nabla m\cdot
    ndS+Dk\int_{\partial\Omega}\beta w^2.
      \end{aligned}
    \end{equation}

Let $R\in(0,R_{0})$ and denote
\[
  \eta(x)=\min\Big\{\frac{R}{2},\ \ d(x)-\frac{d(x)^{2}}{2R}\Big\}.
\]
Hence, $\nabla \eta=(1-d/R)^{+}\nabla d=-(1-d/R)^{+}n$ is Lipschitz
continuous. Direct calculation yields
\begin{equation}\label{eq:3.5}
  \begin{aligned}
    &-\int_{\partial\Omega}\alpha w^{2}(\nabla m\cdot
    n)-Dk\beta w^{2}\\
    &\geq-\int_{\partial\Omega}\alpha w^{2}(\nabla m\cdot
    n)^{+}-Dk\beta w^{2}\\&=\int_{\partial\Omega}
    w^{2}\left[\alpha(\nabla m\cdot
      n)^{+}-Dk\beta\right](\nabla\eta\cdot n)\\
    &=\int_{\Omega}\mathrm{div}\left\{\left[\alpha(\nabla m\cdot
        n)^{+}-Dk\beta\right]w^{2}\nabla\eta\right\}\\
    &=\int_{\Omega}2w(\nabla w\cdot\nabla\eta)\left[\alpha(\nabla
      m\cdot
      n)^{+}-Dk\beta\right]+w^{2}\mathrm{div}\left[\alpha(\nabla
      m\cdot
      n)^{+}-Dk\beta\right]\nabla\eta \\
    &\geq\int_{\Omega}D|\nabla
    w|^{2}+\frac{w^{2}}{D}\left[\left(1-\frac{d}{R}\right)^{+}\right]^{2}\left[\alpha(\nabla
      m\cdot n)^{+}-Dk\beta\right]^{2}\\
    &\ \ \ +\int_{\Omega}w^{2}\mathrm{div}\left\{\left[\alpha(\nabla m\cdot
        n)^{+}-Dk\beta\right]\nabla\eta\right\}.
  \end{aligned}
\end{equation}
Thus we obtain
\begin{align*}
  F(w)&\geq \int_{\Omega}\frac{\alpha^{2}}{D}\left[|\nabla m|^{2}-\left[\left(1-\frac{d}{R}\right)^{+}\right]^{2}(\nabla
        m\cdot n)^{+2}\right]w^{2}\\
      &\ \ \ +\int_{\Omega}\left[2\alpha k\beta(\nabla m\cdot n)^{+}-Dk^{2}\beta^{2}\right]\left[\left(1-\frac{d}{R}\right)^{+}\right]^{2}w^{2}\\
      &\ \ \ +\int_{\Omega}\left\{\alpha\Delta m-\mathrm{div}\left[\left(1-\frac{d}{R}\right)^{+}\left((\nabla
        m\cdot n)^{+}-Dk\beta\right)n\right]\right\}w^2.
\end{align*}

We now further assume that $\beta>0$ on $\partial\Omega$ and
$Dk>\tau>0$ for some given constant $\tau$. Let
$$\beta_{0}=\min_{x\in\partial\Omega}\beta(x), \ \
\tilde{\beta}(x)=\min\Big\{\frac{1}{2}\alpha(\nabla m\cdot
n)^{+},\ \frac{\tau\beta_{0}}{2}\Big\}.$$ Then there exists a constant
$\delta>0$, depending only on $C$, $\beta_{0}$ and $m$, such that
\begin{equation}\label{eq:3.6}
  (\nabla m\cdot
  n)^{+}-\tilde{Dk\beta}\leq \sqrt{1-\delta}(\nabla m\cdot
  n)^{+}.
\end{equation}
As a consequence, we find that
\begin{equation}\label{eq:3.7}
  \begin{aligned}
    &-\int_{\partial\Omega}w^{2}(\alpha(\nabla m\cdot n)^{+}-\beta)\\
    &\geq -\int_{\partial\Omega}w^{2}(\alpha(\nabla m\cdot
    n)^{+}-\beta)\\
    &=-\int_{\partial\Omega}w^{2}(\alpha(\nabla m\cdot
    n)^{+}-\tilde{\beta})(\nabla\eta\cdot n)\\
    &=\int_{\Omega}\mathrm{div}\left((\alpha(\nabla m\cdot
      n)^{+}-\tilde{\beta})\nabla\eta\right)\\
    &=\int_{\Omega}2w(\nabla\eta\cdot\nabla w)(\nabla m\cdot
    n)^{+}-\tilde{\beta})+w^{2}\mathrm{div}((\nabla m\cdot
    n)^{+}-\tilde{\beta})\nabla\eta).
  \end{aligned}
\end{equation}
In view of (\ref{eq:3.6}), we have
\begin{equation}
  \nonumber
  \begin{aligned}
    \int_{\Omega}2w(\nabla\eta\cdot n)(\nabla m\cdot
    n)^{+}-\tilde{\beta})&=\int_{\Omega}-2w\left(1-\frac{d}{R}\right)^{+}(\nabla\eta\cdot
    n)(\nabla m\cdot
    n)^{+}-\tilde{\beta})\\
    &\geq\int_{\Omega}-D|\nabla
    w|^{2}-(1-\delta)\frac{\alpha^{2}w^{2}}{D}\left[\left(1-\frac{d}{R}\right)^{+}\right]^{2}\left[(\nabla
      m\cdot n)^{+}\right]^{2}.
  \end{aligned}
\end{equation}
Combining this with (\ref{eq:3.4}) and (\ref{eq:3.7}), we obtain
\begin{align*}
  F(w)&\geq \int_{\Omega}\frac{\alpha^{2}}{D}\left[|\nabla m|^{2}-(1-\delta)\left[\left(1-\frac{d}{R}\right)^{+}\right]^{2}\left[(\nabla
        m\cdot n)^{+}\right]^{2}\right]w^{2}\\
      &\ \ \ +\int_{\Omega}\left[\alpha\Delta m-\mathrm{div}\left\{\left(1-\frac{d}{R}\right)^{+}(\nabla
        m\cdot n)^{+}n\right]\right\}.
\end{align*}
Therefore we can claim
\begin{lemma}\label{lemma3.1} The following assertions hold.

\begin{itemize}
\item[\rm{(i)}] For any $w\in H^{1}(\Omega)$ and constant $R\in(0,R_{0})$
  \begin{equation}\label{eq:3.9} F(w)\geq
    \int_{\Omega}\left\{\frac{\alpha^{2}}{D}\left[|\nabla
        m|^{2}-\left[\left(1-\frac{d}{R}\right)^{+}\right]^{2}\left[(\nabla
          m\cdot n)^{+}\right]^{2}\right]+A_{R}\right\}w^{2},
  \end{equation}
where
\[
  A_{R}=\alpha\Delta
  m-\mathrm{div}\left\{\left(1-\frac{d}{R}\right)^{+}(\nabla m\cdot
    n)^{+}n\right]+\left[2\alpha k\beta(\nabla m\cdot
    n)^{+}-Dk^{2}\beta^{2}\right]\left[\left(1-\frac{d}{R}\right)^{+}\right]^{2}
\]
satisfying $\|A_{R}\|_{L^{\infty}(\Omega)}\leq \frac{C}{R}$ with
\[
  C=\alpha\|D^{2}m\|_{L^{\infty}(\Omega)}+\|\nabla
  m\|_{L^{\infty}(\Omega)}(1+\|D^{2}d\|_{L^{\infty}(\partial\Omega(R_{0}))}+\alpha
  k\|\beta\|_{L^{\infty}})+k^2\|\beta\|_{L^{\infty}}.
\]

\item[\rm{(ii)}] If $\beta>0$ and $Dk>\tau>0$ for some given constant $
  \tau$, then for any $w\in H^{1}(\Omega)$ and constant $R\in(0,R_{0})$, there exists a constant $\delta>0$
  such that
  \begin{equation}\label{eq:3.10}
    F(w)\geq \int_{\Omega}\left\{\frac{\alpha^{2}}{D}\left[|\nabla m|^{2}-(1-\delta)\left[\left(1-\frac{d}{R}\right)^{+}\right]^{2}\left[(\nabla
          m\cdot n)^{+}\right]^{2}\right]+A_{R}\right\}w^{2},
  \end{equation}
where
\[
  A_{R}=\alpha\Delta
  m-\mathrm{div}\left\{\left(1-\frac{d}{R}\right)^{+}(\nabla m\cdot
    n)^{+}n\right]
  \]
 satisfying $\|A_{R}\|_{L^{\infty}(\Omega)}\leq \frac{C}{R}$ with
  $$
  C=\alpha\|D^{2}m\|_{L^{\infty}(\Omega)}+\|\nabla
  m\|_{L^{\infty}(\Omega)}\left(1+\|D^{2}d\|_{L^{\infty}(\partial\Omega(R_{0}))}\right).
  $$

\end{itemize}

\end{lemma}
This lemma indicates that, as $D\to 0$, the mass of $w^{2}$ is mostly
concentrated on the critical points of $m$ and $m_{\partial\Omega}$;
if $\beta>0$ and $Dk>\tau>0$, the mass of $w^{2}$ is
mostly concentrated on the critical points of $m$.\\

Let $\zeta$ be a smooth function and $W=\zeta w$. As before, multiplying equation
(\ref{eq:3.3}) by $\zeta w$ and integrating over $\Omega$, we obtain
\begin{align*}
  \int_{\Omega}\lambda(D)\zeta^{2}w^{2}
  &=\int_{\Omega}-D\zeta^{2}w\Delta
    w+\frac{\alpha^{2}}{D}|\nabla
    m|^{2}\zeta^{2}w^{2}+\alpha\zeta^{2}w^{2}\Delta
    m+V\zeta^{2}w^{2}\\
  &=\int_{\Omega}D\nabla w\cdot\nabla(\zeta^{2}w)+\frac{\alpha^{2}}{D}|\nabla
    m|^{2}\zeta^{2}w^{2}+\alpha\zeta^{2}w^{2}\Delta
    m+V\zeta^{2}w^{2}-\int_{\partial\Omega}Dw\zeta^{2}\nabla
    w\cdot m \\
  &=\int_{\Omega}D(|\nabla(w\zeta)^{2}-w^{2}|\nabla\zeta|^{2})+\frac{\alpha^{2}}{D}|\nabla
    m|^{2}\zeta^{2}w^{2}-2\alpha\zeta
    w\nabla(\zeta w)\cdot\nabla
    m+V\zeta^{2}w^{2}\\
  &\ \ \ +\int_{\partial\Omega}\alpha w^{2}\zeta^{2}\nabla m\cdot n-Dw\zeta^{2}\nabla
    w\cdot m\\
  &=\int_{\Omega}D\left|\nabla
    W-\frac{\alpha}{D}W\nabla
    m\right|^{2}+VW^{2}-D\int_{\Omega}w^{2}|\nabla\zeta|^{2}+Dk\int_{\partial\Omega}\beta
    w^{2}\zeta^{2}.
\end{align*}
Thus we can conclude the following

\begin{lemma}\label{lemma3.2} Let $(\lambda(D),w)$ be the solution of \eqref{eq:3.3} and $\zeta$ be a
  smooth function. Then $W:=\zeta w$ satisfies
  \[
    \int_{\Omega}\left(D\left|\nabla W-\frac{\alpha}{D}V\nabla
        m\right|^{2}+VW^{2}\right)+Dk\int_{\partial\Omega}\beta
    w^{2}\zeta^{2}=\lambda(D)\int_{\Omega}W^{2}+D\int_{\Omega}w^{2}|\nabla\zeta|^{2}.
  \]
\end{lemma}

Similar to Lemma \ref{lemma3}, we also have
\begin{lemma}\label{lemma3.3} The following assertions hold.

\begin{itemize}
\item[\rm{(i)}] Assume that $B(x_{0},R)\subset\Omega$ and $W=0$ in
  $\Omega\setminus B(x_{0},R)$. Let
  $\kappa_{1}(x_{0}),\cdots,\kappa_{N}(x_{0})$ be the eigenvalues of
  $D^{2}m(x_{0})$. Then for $C=2N\|D^{3}m\|_{L^{\infty}(\Omega)}$,
  \begin{equation}\nonumber
    F(W)\geq\alpha\left[\sum_{i=1}^{N}\left(|\kappa_{i}(x_{0})|+\kappa_{i}(x_{0})\right)-CR\right]\int_{\Omega}W^{2}.
  \end{equation}

\item[\rm{(ii)}] Assume that $x_{0}\in\partial\Omega$, $W=0$ in
  $\Omega\setminus B(x_{0},R)$, $|\nabla m(x_{0})|=0$,
  $\det(D^{2}m(x_{0}))\neq 0$ and $n(x_{0})$ is an eigenvector of $D^{2}m(x_{0})$. Let
  $\kappa_{1}(x_{0}),\cdots,\kappa_{N}(x_{0})$ be the eigenvalues of
  $D^{2}m(x_{0})$. Then
  \begin{equation}\nonumber
    F(W)\geq\alpha\left[\sum_{i=1}^{N}\left(|\kappa_{i}(x_{0})|+\kappa_{i}(x_{0})\right)-C(1+k^{2})(R+D)\right]\int_{\Omega}W^{2},
  \end{equation}
  where the positive constant $C$ depends only on $m$, $\beta$ and $\partial\Omega$.

\item[\rm{(iii)}] Assume that $x_{0}\in\partial\Omega$ and $W=0$ in
  $\Omega\setminus B(x_{0},R)$, and $|\nabla m(x_{0})|=\nabla
  m(x_{0})\cdot n(x_{0})>0$. Let  $\kappa_{N}(x_{0}):=0$ and
  $\kappa_{1}(x_{0}),\cdots,\kappa_{N-1}(x_{0})$ be the eigenvalues of
  $D^{2}m_{\partial\Omega}(x_{0})$.  Then
\begin{equation}\nonumber
  F(W)\geq\alpha\left[\sum_{i=1}^{N}\left(|\kappa_{i}(x_{0})|+\kappa_{i}(x_{0})\right)+2k\beta(x_0)\nabla m(x_0)\cdot n(x_{0})-Ck(R+D)\right]\int_{\Omega}W^{2},
\end{equation}
where the positive constant $C$ depends only on $m$, $\beta$ and $\partial\Omega$.

\end{itemize}
\end{lemma}
\begin{proof} The assertion (i) follows from Lemma \ref{lemma3} since $W=0$ on $\partial\Omega$.

 We next prove (ii). There exists an orthomormal basis
  $\{\tau_{1}(x_{0}),\cdots,\tau_{N-1}(x_{0}),n(x_{0})\}$ under which
  \[
    D^{2}(x_{0})=\mathrm{diag}(\kappa_{1}(x_{0}),\cdots,\kappa_{N-1}(x_{0}),\kappa_{N}(x_{0})),
  \]
  where $\kappa_{1}(x_{0}),\cdots,\kappa_{N-1}(x_{0})$ are the
  eigenvalues of $D^{2}m_{\partial\Omega}(x_{0})$. By translation and
  rotation we may assume that $x_{0}=0,n(x_{0})=(0,\cdots,0,-1)$ and
  $\tau_{1}(x_{0}),\cdots,\tau_{N-1}(x_{0})$ are exactly the
  $x_{1},\cdots,x_{N-1}$ coordinate direction respectively. We use the
  change of variable $z=Z(x)$ introduced in \cite[Section
  3.4]{CL2}, which maps $\partial\Omega\cap B(x_{0},R)$ to the
  set $z=0$ on the $z$ space.

  Let $x=X(z)$ be the inverse of
  $z=Z(x)$. Then $\frac{\partial Z}{\partial x}=I+O(|x|)$. Moreover,
  \[
    M(z):=m(X(z))=\sum_{i=1}^{N}\frac{\kappa_{i}(x_{0})}{2}z_{i}^{2}+O(|z|^{3}),\ \
    M_{z_{N}}(z',0)=0,\ \ \forall|z'|<3R.
  \]
  Denote $\tilde{m}(z)=m(X(z))$, $\tilde{W}(z)=W(X(z))$ and
  $\tilde{\beta}(z)=\beta(X(z))$. Then we have
  \begin{align*}
    F(W)&=\int_{\Omega}|\nabla W-\frac{\alpha W}{D}\nabla m|^{2}+Dk
          \int_{\partial\Omega}\beta W^{2}\\
        &=(1+O(R))\left\{\int_{\mathbb{R}^{N}_{+}}D\Big|\nabla \tilde{W}-\frac{\alpha \tilde{W}}{D}\nabla \tilde{m}\Big|^{2}dz+Dk
          \int_{\mathbb{R}^{N-1}}\tilde{\beta}\tilde{W}^{2}dz'\right\}\\
        &=(1+O(R))\left\{\int_{\mathbb{R}^{N}_{+}}\Bigg[ D\Big|\nabla
          \tilde{W}-\frac{\alpha \tilde{W}}{D}\nabla
          \tilde{m}\Big|^{2}-Dk\tilde{\beta}\tilde{W}\partial_{z_{N}}\tilde{W}-Dk\partial_{z_{N}}\tilde{\beta}W^{2}\Bigg]dz\right\}\\
        & \geq (1+O(R))\left(J+K-DkC\int_{\mathbb{R}^{N}_{+}}W^{2}dz\right),
  \end{align*}
  where
  \[
    J=\int_{\mathbb{R}^{N}_{+}}D\sum_{i=1}^{N-1}|\partial_{z_{i}}
    \tilde{W}-\frac{\alpha \tilde{W}}{D}\partial_{z_{i}}
    \tilde{m}|^{2}dz,
  \]
  \[
    K=\int_{\mathbb{R}^{N}_{+}}D\Big|\partial_{z_{N}}
    \tilde{W}-\frac{\alpha \tilde{W}}{D}\partial_{z_{N}}
    \tilde{m}\Big|^{2}-Dk\tilde{\beta}\tilde{W}\partial_{z_{N}}\tilde{W}dz.
  \]
  We can calculate
  \begin{align*}
    J&\geq
       -\sum_{i<N,\kappa_{i}>0}\int_{\mathbb{R}^{N}_{+}}2\alpha\partial_{x_{i}}\tilde{W}^{2}\partial_{x_{i}}\tilde{m}dz\\
     &=\alpha\sum_{i<N,\kappa_{i}>0}\int_{\mathbb{R}^{N}_{+}}\left[2\kappa_{i}+2e_{i}(D^{2}m(x)-D^{2}m(x_{0}))\right]\tilde{W}^{2}\\
     &\geq
       \alpha\left[\sum_{i=1}^{N-1}(|\kappa_{i}|+\kappa_{i})-2N\|D^{3}m\|_{L^{\infty}(\Omega)}R\right]\int_{\mathbb{R}^{N}_{+}}\tilde{W}^{2}.
  \end{align*}
  If $\kappa_{N}\leq 0$,
  \begin{align*}
    K&=\int_{\mathbb{R}^{N}_{+}}D\Big|\partial_{z_{N}}
       \tilde{W}-\frac{\alpha \tilde{W}}{D}\partial_{z_{N}}
       \tilde{m}\Big|^{2}-Dk\tilde{\beta}\tilde{W}\partial_{z_{N}}\tilde{W}dz\\
     &=\int_{\mathbb{R}^{N}_{+}}D\Big|\partial_{x_{N}}\tilde{W}-\frac{\alpha\tilde{W}}{D}\partial_{x_{N}}\tilde{m}-\frac{1}{2}k\tilde{\beta}\tilde{W}\Big|^{2}-\frac{1}{4}Dk^{2}\tilde{\beta}^{2}\tilde{W}^{2}+k\alpha\tilde{\beta}\tilde{W}^{2}\partial_{x_{N}}\tilde{m}dz\\
     &\geq\int_{\mathbb{R}^{N}_{+}}-(Dk^{2}\|\tilde{\beta}\|_{L^{\infty}}^{2}+k\alpha\|\tilde{\beta}\|_{L^{\infty}}\|D^{2}\tilde{m}\|_{L^{\infty}}R)\tilde{W}^{2}dz.
  \end{align*}
  Otherwise if $\kappa_{N}>0$,
  \begin{align*}
    K&=\int_{\mathbb{R}^{N}_{+}}D\Big|\partial_{z_{N}}
       \tilde{W}-\frac{\alpha \tilde{W}}{D}\partial_{z_{N}}
       \tilde{m}\Big|^{2}-Dk\tilde{\beta}\tilde{W}\partial_{z_{N}}\tilde{W}dz\\
     &=\int_{\mathbb{R}^{N}_{+}}D\Big|\partial_{x_{N}}\tilde{W}+\frac{\alpha\tilde{W}}{D}\partial_{x_{N}}\tilde{m}-\frac{1}{2}k\tilde{\beta}\tilde{W}\Big|^{2}-2\alpha\partial_{x_{N}}\tilde{W}^{2}\partial_{x_{N}}m-\frac{1}{4}Dk^{2}\tilde{\beta}^{2}\tilde{W}^{2}-k\alpha\tilde{\beta}\tilde{W}^{2}\partial_{x_{N}}\tilde{m}dz\\
     &\geq
       2\alpha\int_{\mathbb{R}^{N}_{+}}\tilde{W}^{2}\partial^{2}_{x_{N}x_{N}}\tilde{m}dz-\int_{\mathbb{R}^{N}_{+}}(Dk^{2}\|\tilde{\beta}\|_{L^{\infty}}^{2}+k\beta\|\tilde{\beta}\|_{L^{\infty}}\|D^{2}\tilde{m}\|_{L^{\infty}}R)\tilde{W}^{2}|\partial_{x_{N}}\tilde{m}|dz\\
     &\geq
       \int_{\mathbb{R}^{N}_{+}}\left[2\alpha\kappa_{N}-(Dk^{2}\|\tilde{\beta}\|_{L^{\infty}}^{2}
       +|D^{3}\tilde{m}\|_{L^{\infty}}R+k\alpha\|D^{2}\tilde{m}\|_{L^{\infty}}R)\right]\tilde{W}^{2}dz.
  \end{align*}
 Therefore, it holds
  \[
    F(W)\geq\alpha\left[\sum_{i=1}^{N}\left(|\kappa_{i}(x_{0})|+\kappa_{i}(x_{0})\right)-C(1+k^{2})(R+D)\right]\int_{\Omega}W^{2},
  \]
  for some constant $C$ depending only on $m$, $\beta$ and $\partial\Omega$.

  Lastly we prove (iii). Let $\{\tau_{1}(x_{0}),\cdots,\tau_{n-1}(x_{0}),n(x_{0})\}$ be
  an orthonormal eigenbasis of $D^{2}m_{\partial\Omega}(x_{0})$
  associated with eigenvalues
  $\{\kappa_{1}(x_{0}),\cdots,\kappa_{N}(x_{0})\}$ with
  $\kappa_{N}=0$. By the Schmidt process from the set
  $\{\tau_{1}(x_{0}),\cdot,\tau_{N-1}(x_{0}),n(x)\}$, we obtain obtain
  an orthonormal basis
  $\{\tau_{1}(x),\cdot,\tau_{N-1}(x),\tau_{N}(x)\}$ in
  $B(x_{0},R)\cap\bar{\Omega}$ with $\tau_{N}(x)=n(x)$. Then we have
  \begin{align*}
    F(W)&=\sum_{i=1}^{N}\int_{\Omega}D\left|\tau_{i}\cdot\left(\nabla
          W-\frac{\alpha W}{D}\nabla m\right)\right|^{2}+Dk\int_{\partial\Omega}\beta W^2\\
        &=\sum_{i=1}^{N-1}\int_{\Omega}D\left|\tau_{i}\cdot\nabla
          W+\mathrm{sgn}(\kappa_{i})\frac{\alpha W}{D}\nabla
          m\right|^{2}\\
        &\ \ \ -\sum_{\kappa_{i}>0}\int_{\Omega}2\alpha(\tau_{i}\cdot\nabla
          W^{2})(\tau_{i}\cdot\nabla m)\\
        &\ \ \ +\int_{\Omega}D\left|\tau_{N}\cdot\left(\nabla
          W-\frac{\alpha W}{D}\nabla m\right)\right|^{2}+Dk\int_{\partial\Omega}\beta W^2\\
        &\geq J+K,
    \end{align*}
    where
    \begin{align*}
      J:&=-\sum_{\kappa_{i}>0}\int_{\Omega}2\alpha(\tau_{i}\cdot\nabla
          W^{2})(\tau_{i}\cdot\nabla m)\\
        &=2\alpha\sum_{\kappa_{i}>0}\int_{\Omega}-\mathrm{div}(W^{2}(\tau_{i}\cdot\nabla
          m)\tau_{i})+W^{2}e_{i}D^{2}m e_{i}^{T}\\
        &=-2\alpha\sum_{i<N,\kappa_{i}>0}\int_{\partial\Omega\cap B(x_{0},R)}W^{2}(\tau_{i}\cdot\nabla
          m)\tau_{i}\cdot\tau_{N}+2\alpha\sum_{\kappa_{i}>0}\int_{\Omega}W^{2}\tau_{i}D^{2}m
          \tau_{i}^{T}\\
        &=2\alpha\sum_{i<N,\kappa_{i}>0}\int_{\Omega}W^{2}\tau_{i}D^{2}m \tau_{i}^{T},
    \end{align*}
    and
    \begin{align*}
      K:&=\int_{\Omega}D\left|\tau_{N}\cdot\left(\nabla
          W-\frac{\alpha W}{D}\nabla m\right)\right|^{2}+Dk\int_{\partial\Omega}\beta W^2\\
        &= \int_{\Omega}D|\tau_{N}\cdot\nabla
          W|^2+\frac{\alpha^2 W^2}{D}(\tau_{N}\cdot\nabla m)^2\\
        &\ \ \ -2\alpha W(\tau_{N}\cdot\nabla m)(\tau_{N}\cdot\nabla W)+Dk\int_{\Omega}\tau_{N}\cdot\nabla(\beta W^2)\\
        &=\int_{\Omega}D|\tau_{N}\cdot\nabla
          W|^2+\frac{\alpha^2 W^2}{D}(\tau_{N}\cdot\nabla m)^2\\
        &\ \ \ -\int_{\Omega}2(\tau_{N}\cdot\nabla W)[\alpha
          W(\tau_{N}\cdot\nabla m)-Dk\beta W]+Dk\int_{\Omega}
          W^2\tau_{N}\cdot\nabla \beta \\
        &=\int_{\Omega}\frac{\alpha^2 W^2}{D}(\tau_{N}\cdot\nabla m)^2
          -\frac{1}{D}\int_{\Omega}[\alpha W(\tau_{N}\cdot\nabla m)-Dk\beta W]^{2}+Dk\int_{\Omega} W^2\tau_{N}\cdot\nabla \beta \\
        &\geq\int_{\Omega}2\alpha k\beta\tau_{N}\cdot\nabla m W^{2}-Dk\beta^2W^2+Dk\tau_{N}\cdot\nabla \beta W^2.
    \end{align*}
    Therefore,
    \begin{align*}
      F(W)&\geq 2\alpha\sum_{\kappa_{i}>0}\int_{\Omega}W^{2}\tau_{i}D^{2}m \tau_{i}^{T}+\int_{\Omega}2\alpha k\beta\tau_{N}\cdot\nabla m W^{2}-Dk\beta^2W^2+Dk\tau_{N}\cdot\nabla \beta W^2\\
          &\geq\alpha \left[\sum_{i=1}^{N}(|\kappa_{i}|+\kappa_{i})+2k\beta(x_0)\nabla m(x_0)\cdot n(x_{0})-2N(\|m\|_{C^{3}(\bar{\Omega})}+\|\beta\|_{C^{1}(\bar{\Omega})})R\right]\int_{\Omega}W^{2}\\
          &\ \ \ -Dk[1+\|\beta\|_{C^{1}(\bar{\Omega})}]^{2}\int_{\Omega}W^{2}.
    \end{align*}
    The assertion (iii) follows.
  \end{proof}
  \begin{remark}\label{r1} If $\beta\geq 0$, it is easily checked that the estimate (ii) of Lemma \ref{lemma3.3} can be replaced by
    \[
      F(W)\geq\alpha\left[\sum_{i=1}^{N}\left(|\kappa_{i}(x_{0})|+\kappa_{i}(x_{0})\right)-CR\right]\int_{\Omega}W^{2}.
    \]
  \end{remark}

We are now in a position to give
 \vskip6pt
\noindent
 {\bf Proof of Theorem \ref{th1.2}:}
  By a similar argument as in Theorem \ref{theorem1} and Lemma
  \ref{lemma3.1}, we have
  \begin{align*}
    \liminf_{D\to 0}\lambda(D)&\geq\min\Big\{\min_{x\in\Sigma_{1}\cup\Sigma_{2}}\big\{V(x)+\alpha\sum_{i=1}^N(|\kappa_i(x)|+\kappa_i(x))\big\},\\
                                   &\ \ \ \min_{x\in\Sigma_{3}}\big\{V(x)+\alpha\sum_{i=1}^N(|\kappa_i(x)|+\kappa_i(x))+2\alpha k \beta(x)|\nabla m(x)|\big\}\Big\}=:\Lambda.
  \end{align*}
  It remains to show that
  \begin{equation}\label{th1.2-ab}\limsup_{D\to 0}\lambda(D)\leq \Lambda.
    \end{equation}
  It follows from Theorem \ref{theorem1} that
\begin{equation}\label{th1.2-aa}
    \lim_{D\to 0}\lambda(D)\leq\min_{x\in\Sigma_{1}}\Big\{V(x)+\alpha\sum_{i=1}^N(|\kappa_i(x)|+\kappa_i(x))\Big\}.
  \end{equation}

  Let $x_{0}\in \Sigma_{2}$. By a translation and rotation we may
  assume $x_{0}=0$, $n(x_{0})=(0,\cdots,0,-1)$ and
  $D^{2}m(x_{0})= \mathrm{diag}(\kappa_{1},\cdots,\kappa_{N})$. Using
  the test function in the proof of \cite[Theorem 8.1]{CL2}, we
  can show that
\begin{equation}\label{th1.2-a}
  \lim_{D\to 0}\lambda(D)\leq V(x_{0})+\alpha\sum_{i=1}^N(|\kappa_i(x_{0})|+\kappa_i(x_{0})).
\end{equation}
Fix a small positive constant $\delta$ and define
\[
  \zeta(x)=\frac{1}{\epsilon^{N/2}}\exp\left(-\frac{1}{2\epsilon^{2}}\sum_{i=1}^{N}(|\kappa_{i}|+\delta)x_{i}^{2}\right),
\]
where $\epsilon=\sqrt{\frac{D}{\alpha}}$.

Let
$w(x)=\frac{\zeta(x)}{\sqrt{c_{\epsilon,\delta}}}$ with
$c_{\epsilon,\delta}=\int_{\Omega}\zeta^{2}(x)$. One easily observes that
\[
  \lim_{\epsilon\to
    0}c_{\epsilon,\delta}=\int_{\mathbb{R}_{+}^{N}}\exp\left(-\sum_{i=1}^{N}(|\kappa_{i}|+\delta)y_{i}^{2}\right)dy=\frac{1}{2}\prod_{i=1}^{N}\sqrt{\frac{\pi}{|\kappa_{i}|+\delta}}=:c(\delta).
\]
Notice that
\[
  \lim_{\epsilon\to
    0}\epsilon\int_{\partial\Omega}w^{2}=\frac{1}{c(\delta)}\int_{\mathbb{R}^{N-1}}\exp\left(-\sum_{i=1}^{N-1}(|\kappa_{i}|+\delta)y_{i}^{2}\right)dy.
\]
Thus
\[
  \lim_{D\to 0}Dk\int_{\partial\Omega}\beta w^{2}=0.
\]
Letting $\epsilon\to 0$ first and then $\delta\to 0$ infers \eqref{th1.2-a}.

We may also use the test function in
\cite[Theorem 8.1]{CL2} to show
\begin{equation}\label{th1.2-ac}
  \lim_{D\to 0}\lambda(D)\leq
  \min_{x\in\Sigma_{3}}\Big\{V(x)+\alpha\sum_{i=1}^N(|\kappa_i(x)|+\kappa_i(x))+2\alpha k \beta(x))|\nabla m(x)|\Big\},
\end{equation}
Let $x_0\in\Sigma_3$, $b=\nabla m(x_{0})\cdot n(x_{0})>0$. By
translation and rotation, we may assume that $x_{0}=0$ and
$n(x_{0})=(0,\cdots,0,-1)$. Near $x_{0}$, the boundary
$\partial\Omega$ can be expressed as
\[
  x_{N}=\phi(x'), \mathrm{ where\ }
  x'=(x_{1},\cdots,x_{N-1}),\phi(0')=0,\nabla_{x'}\phi(0')=0'.
\]
Let $z=x_{N}-\phi(x')$, then
\begin{align*}
  m(x)&=m(x',\phi(x')+z)=m(x',\phi(x'))+zm_{x_{N}}(x',\phi(x'))+O(z^{2})\\
      &=m_{\partial\Omega}(x')+(x_N-\phi(x'))m_{x_N}(x',\phi(x'))+O(z^2),
\end{align*}
\begin{align*}
  m_{x_{i}}&=\partial_{x_{i}}m_{\partial\Omega}(x')-\phi_{x_{i}}(x')m_{x_{N}}(x_{0})+O(|z|+|x'|^{2})\\
           &=\kappa_{i}x_{i}+b\phi_{x_{i}}(x')+O(|z|+|x'|^{2}), \ \  \forall i=1,\cdots,N-1,
\end{align*}
\begin{align*}
  m_{x_{N}}&=m_{x_{N}}(x',\phi(x'))+O(|z|)\\
           &=-b+\sum_{i=1}^{N-1}a_{i}x_{i}+O(|z|+|x'|^{2}),
\end{align*}
where $m_{\partial\Omega}$ is the restriction of $m(x)$ on
$\partial\Omega$ and $a_{i}=m_{x_{i}x_{N}}(x_{0})$.

Let
$\epsilon=\sqrt{D/\alpha}$ and
$M=\delta+\frac{1}{2}\sum_{i=1}^{N}a_{i}^{2}/[|\kappa_{i}|+\delta]$,
where $\delta>0$ is a small constant. In
$B(x_{0},\epsilon)\cap\bar{\Omega}$, we define
\[
  \zeta(x)=\exp\left(-\frac{1}{\epsilon^{2}}\left[bz+\frac{1}{2}\sum_{i=1}^{N-1}[(|\kappa_{i}|+\delta)x_{i}^{2}-2a_{i}x_{i}z]+Mz^{2}\right]\right),
\]
with $z=\phi(x')$.  Notice that
$\zeta=O(e^{-\frac{\delta}{\epsilon}})$ on
$\partial B(x_{0},\sqrt{\epsilon})\cap\bar{\Omega}$. We may extend $w$ to $\bar{\Omega}$
such that $\zeta=O(e^{-\frac{\delta}{\epsilon}})$,
$|\nabla \zeta|=O(e^{-\frac{\delta}{\epsilon}})$ in
$\bar{\Omega}\setminus B(x_0,\sqrt{\epsilon})$.

Set
$w(x)=\frac{\zeta(x)}{\sqrt{\int_{\Omega}\zeta^{2}(x)}}$. It then follows that
\[
  \lambda(D)\leq \int_{\Omega}D|\nabla w-\frac{\alpha}{D}w\nabla
  m|^{2}+Vw^{2}+Dk\int_{\partial\Omega}\beta w^2.
\]
Since $\zeta=O(e^{-\frac{\delta}{\epsilon}})$, $|\nabla \zeta|=O(e^{-\frac{\delta}{\epsilon}})$ in $\bar{\Omega}\setminus B(x_0,\sqrt{\epsilon})$,
\[
  \int_{\Omega\setminus B(x_0,\sqrt{\epsilon})}D|\nabla
  w-\frac{\alpha}{D}w\nabla
  m|^{2}+Vw^{2}=O(e^{-\frac{\delta}{\epsilon}}).
\]
Let $y'=x'/\epsilon$ and $t=z/\epsilon^{2}$. We have
\[
  \int_{\Omega\setminus B(x_0,\sqrt{\epsilon})}D|\nabla
  w-\frac{\alpha}{D}w\nabla m|^{2}+Vw^{2}\leq
  \frac{\int_{0}^{\infty}\int_{\mathbb{R}^{N-1}}B(y',t)e^{-A(y',t)}dy'dt}{\int_{0}^{\infty}\int_{\mathbb{R}^{N-1}}e^{-A(y',t)}dy'dt},
\]
where
\[
  A(y',t)=2bt+\sum_{i=1}^{N-1}[|\kappa_{i}|+\delta]y_{i}^{2}-2\epsilon
  t a_{i}y_{i}+2M\epsilon^{2}t^{2},
\]
\[
  B(y',t)=V(x_{0})+O(\epsilon^{2}t+\epsilon
  |y'|)+\alpha\sum_{i=1}^{N-1}[|\kappa_{i}|+\delta+\kappa_{i}]^{2}y_{i}^{2}+O(\alpha)[t^{2}\epsilon^{2}+\epsilon
  t|y'|+\epsilon |y'|^{3}].
\]
Passing to the limit, we obtain
\[
  \limsup_{D\to 0}\int_{\Omega}D|\nabla w-\frac{\alpha}{D}w\nabla
  m|^{2}+Vw^{2}\leq V(x_0)+\frac{\alpha}{2}\sum_{i=1}^{N-1}
  \frac{[|\kappa_i|+\delta+\kappa_i]^2}{|\kappa_i|+\delta}.
\]
On the other hand, as $D\to 0$,
\[
  D\int_{\partial\Omega}\beta w^2\to
  \alpha\frac{\int_{\mathbb{R}^{N-1}}\beta(x_0)e^{-\sum_{i=1}^{N-1}(|\kappa_i|+\delta)y_i^2}dy'}{\int_{0}^{\infty}\int_{\mathbb{R}^{N-1}}e^{-2bt-\sum_{i=1}^{N-1}(|\kappa_i|+\delta)y_i^2}dy'dt}=2\alpha
  b\beta(x_0).
\]
Hence
\[
  \limsup_{D\to 0}\lambda(D)\leq
  V(x_0)+\frac{\alpha}{2}\sum_{i=1}^{N-1}
  \frac{[|\kappa_i|+\delta+\kappa_i]^2}{|\kappa_i|+\delta}+2\alpha k
  b\beta(x_0).
\]
Sending $\delta\to 0$ results in \eqref{th1.2-ac}.

A combination of \eqref{th1.2-aa}, \eqref{th1.2-a} and \eqref{th1.2-ac} implies \eqref{th1.2-ab}.
The proof is thus complete. {\hfill $\Box$}

\vskip10pt
The following result concerns the situation that $\beta(x)>0$ for all $x\in\partial\Omega$, $D\to 0$,
$k\to\infty$ and $\liminf kD>\tau>0$. 

\begin{thm} Assume that $\beta>0$ on $\partial\Omega$. The following assertions hold.

\begin{itemize}
\item[\rm{(i)}] If $\Sigma_{1}\cup\Sigma_{2}=\emptyset$, then
  \[
  \lim_{D\to 0,k\to\infty,\liminf
    Dk>0}\lambda(D)=+\infty.
\]

 \item[\rm{(ii)}] If $\Sigma_{1}\cup\Sigma_{2}\not=\emptyset$, and further assume that for all
  $x\in\Sigma_{2}$, $n(x)$ is an eigenvector of $D^{2}m(x)$ with
  the corresponding eigenvalue $\kappa_{N}(x)= 0$, then
\[
  \lim_{D\to 0,k\to\infty,\liminf
    Dk>0}\lambda(D)=\min_{x\in\Sigma_{1}\cup\Sigma_{2}}\Big\{V(x)+\alpha\sum_{i=1}^N(|\kappa_i(x)|+\kappa_i(x))\Big\},
\]
where $\kappa_N(x)=0$, $\kappa_1(x), \kappa_2(x),\cdots$ are eigenvalues of
$D^2m(x)$.

\end{itemize}
\end{thm}
\begin{proof} The assertion (i) follows directly from Lemma \ref{lemma3.1}(ii).
  The proof of the assertion (ii) is similar to that of Theorem \ref{theorem1}(ii).
   Indeed, the test function in the proof of Theorem \ref{theorem1} can be used to derive the upper bound, while the
  lower bound follows using Lemma \ref{lemma3.1}(ii), Lemma \ref{lemma3.3}(i) and Remark \ref{r1}.
  The details are omitted here.
\end{proof}

\begin{remark}\label{r2.2}
More generally, we can consider the following eigenvalue problem
\begin{equation}\label{general-p}
\begin{cases}
  -D\Delta \phi -2\alpha\nabla m(x)\cdot \nabla\phi+V(x)\phi=\mu\phi &\hbox{ in }\Omega,\\
  \phi=0& \hbox{ on } \Gamma_1\\
  \frac{\partial\phi}{\partial n}+k_{1}\beta_{1}(x)\phi=0 &\hbox{ on
  }\Gamma_2,\\
  \frac{\partial\phi}{\partial n}+k_{2}\beta_{2}(x)\phi=0 &\hbox{ on
  }\Gamma_3,
\end{cases}
\end{equation}
where $\Gamma_1,\Gamma_2,\Gamma_{3}$ is a disjoint union of
$\partial\Omega$ with $\Gamma_1\cup\Gamma_2\cup\Gamma_{3}=\partial\Omega$, $\Gamma_1$ or $\Gamma_{3}$ may be empty, and
$\beta_{i}\in C(\bar{\Omega})\,(i=2,3)$ is a given function such that
$\beta_{i}(x)=0$ for $x\in\partial\Omega\setminus\Gamma_{i}$.

The principal eigenvalue $\lambda(D)$ is characterized by
\begin{equation}\nonumber
  \begin{aligned}
    \lambda(D)&=\inf_{\phi\in W^{1,2}(\Omega),\phi=0\hbox{ on
      }\Gamma_1,\phi\not\equiv 0}\frac{\int_{\Omega}e^{2\alpha
        m/D}(D|\nabla\phi|^2+V\phi^2)+D\int_{\partial\Omega}e^{2\alpha
        m/D}(k_{1}\beta_{1}+k_{2}\beta_{2})\phi^2}{\int_{\Omega}e^{2\alpha
        m/D}\phi^2}\\
    &=\inf_{w\in W^{1,2}(\Omega), \phi=0\hbox{ on }\Gamma_1,
      \int_{\Omega}w^{2}=1}\int_{\Omega}D\Big|\nabla
    w-\frac{\alpha}{D}w\nabla
    m\Big|^{2}+Vw^{2}+D\int_{\partial\Omega}(k_{1}\beta_{1}+k_{2}\beta_{2}) w^2.
  \end{aligned}
\end{equation}

Let
$\Sigma_1=\{x\in\Omega:\ |\nabla m(x)|=0\},
\Sigma_2'=\{x\in\Gamma_{1}:\ |\nabla m(x)|=0\}$,
$\Sigma_3'=\{x\in\Gamma_{2}:\ |\nabla m(x)|=0\}$,
$\Sigma_4'=\{x\in\Gamma_{3}:\ |\nabla m(x)|=0\}$ and
$\Sigma_5'=\{x\in\Gamma_{3}:\ |\nabla m(x)|=\nabla m(x)\cdot
n(x)>0\}$.

Assume that $n(x)$ is an eigenvector of $D^{2}m(x)$ with
the corresponding eigenvalue $\kappa_{N}(x)=0$ for all
$x\in\Sigma_{2}'\cup\Sigma_{3}'$, and $\mathrm{det}(D^{2}m(x))\neq 0$ and
$n(x)$ is an eigenvector of $D^{2}m(x)$ for all $x\in\Sigma_{4}'$. Then by the similar analysis as before,
we can assert that

\begin{itemize}
\item[\rm{(i)}] if $\Sigma_{1}\cup\Sigma_{2}'\cup\Sigma_{3}'\cup\Sigma_{4}'\cup\Sigma_{5}'=\emptyset$, then
  $\lim_{D\to 0}\lambda(D)=+\infty$.

 \item[\rm{(ii)}] if $\Sigma_{1}\cup\Sigma_{2}'\cup\Sigma_{3}'\cup\Sigma_{4}'\cup\Sigma_{5}'\not=\emptyset$, then
\begin{align*}
  \lim_{D\to 0}\lambda(D)&=\min\Big\{\min_{x\in\Sigma_{1}\cup\Sigma_{2}'\cup\Sigma_{3}'
                                \cup\Sigma_{4}'}\big\{V(x)+\alpha\sum_{i=1}^N(|\kappa_i(x)|+\kappa_i(x))\big\},\\
                              &\ \ \ \min_{x\in\Sigma_{5}'}\big\{V(x)+\alpha\sum_{i=1}^N(|\kappa_i(x)|+\kappa_i(x))+2\alpha k \beta(x_{0})|\nabla m(x_0)|\big\}\Big\}=:\lambda^*,
\end{align*}
where, when $x\in\Sigma_{1}\cup\Sigma_{2}'\cup\Sigma_{3}'\cup\Sigma_{4}'$,
$\kappa_1(x), \kappa_2(x),\cdots,\kappa_N$ are the eigenvalues of
$D^2m(x)$; when $x\in\Sigma_{5}'$, $\kappa_{N}(x)=0$ and
$\kappa_1(x), \kappa_2(x),\cdots,\kappa_{N-1}(x)$ are the eigenvalues of
$D^2m_{\partial\Omega}(x)$.

\end{itemize}

If we further assume that $\beta_{1}>0$, then it holds
\begin{itemize}
\item[\rm{(i)}] if $\Sigma_{1}\cup\Sigma_{2}'\cup\Sigma_{3}'\cup\Sigma_{4}'\cup\Sigma_{5}'=\emptyset$, then
  $\lim_{D\to 0,k_{1}\to +\infty,\liminf Dk_{1}>0}\lambda(D)=+\infty$.

 \item[\rm{(ii)}] if $\Sigma_{1}\cup\Sigma_{2}'\cup\Sigma_{3}'\cup\Sigma_{4}'\cup\Sigma_{5}'\not=\emptyset$, then
$\lim_{D\to 0,k_{1}\to +\infty,\liminf Dk_{1}>0}\lambda(D)=\lambda^*$,
where the definition of $\kappa_i$ is as above.

\end{itemize}

\end{remark}

%%%%%%%%%%%%%%%%%%%%%%%%%%%%%%%%%%%%%%%%%%%%%%%%%%%%%%%%%%%%%%%%%%%%%%%%%%%%%%%%%%%%%%%%%%%%%%%%%%%%%%%%%%%%%%%%%%%%%%%%%%%%%%%%%%%%%%%%%%%%%%%%%%%%%%%%%%%%%%%%%%%%%%%%%%%

\section{Asymptotic behavior as $D\to\infty$:\ Proof of Theorem \ref{th1.3}}
This section is devoted to the proof of Theorem \ref{th1.3}. To this aim, we recall the following trace theorem; see, for instance, \cite[Thereom 1.5.1.10]{Gr}.

\begin{lemma}
\label{lemma4.1}
Given any positive number $\epsilon$, there exists a constant
$C(\epsilon)$ such that
\[
  \int_{\partial\Omega}u^{2}\leq \epsilon\int_{\Omega}|\nabla
  u|^{2}+C(\epsilon)\int_{\Omega}u^{2}
\]
holds for any $u\in H^{1}(\Omega)$.
\end{lemma}

 \vskip6pt
\noindent
 {\bf Proof of Theorem \ref{th1.3}:} We first note that
   \[
    \mu_1=\min_{\phi\in
      W^{1,2}(\Omega),\int_{\Omega}\phi^{2}=1}\left\{\int_{\Omega}|\nabla\phi|^{2}+\int_{\partial\Omega}\beta\phi^{2}\right\}
  \]
  and
  \begin{equation*}
  \begin{aligned}
    \lambda(D)&=\min_{\phi\in
      W^{1,2}(\Omega)}\frac{\int_{\Omega}e^{2\alpha
        m/D}(D|\nabla\phi|^{2}+V\phi^{2})+D\int_{\partial\Omega}e^{2\alpha
        m/D}\beta\phi^{2}}{\int_{\Omega}e^{2\alpha
        m/D}\phi^{2}}\\
    &=\inf_{w\in W^{1,2}(\Omega),
      \int_{\Omega}w^{2}=1}\left\{\int_{\Omega}D\Big|\nabla w-\frac{\alpha}{D}w\nabla m\Big|^{2}+Vw^{2}+D\int_{\partial\Omega}\beta w^2\right\}. \\
  \end{aligned}
\end{equation*}

We now verify (i). If $\mu_1>0$, there holds
\[
  \min_{\phi\in
    W^{1,2}(\Omega),\int_{\Omega}\phi^{2}=1}\left\{\int_{\Omega}|\nabla\phi|^{2}+\int_{\partial\Omega}\beta\phi^{2}\right\}=\mu_1>0.
\]
Let $\phi_{D}$ be the normalized eigenfunction of (\ref{eq:4.1}) corresponding to $\lambda(D)$ with
$\int_{\Omega}\phi_{D}^{2}=1$. Then
\[
  \int_{\Omega}|\nabla\phi_{D}|^{2}+\int_{\partial\Omega}\beta\phi_{D}^{2}\geq\mu>0.
\]
By Lemma \ref{lemma4.1}, for any given $\epsilon>0$ there exists a
constant $C(\epsilon)$ such that
\[
  \int_{\partial\Omega}\phi_{D}^{2}\leq \epsilon\int_{\Omega}|\nabla
  \phi_{D}|^{2}+C(\epsilon)\int_{\Omega}\phi_{D}^{2}=\epsilon\int_{\Omega}|\nabla
  \phi_{D}|^{2}+C(\epsilon).
\]
Thus it holds
\begin{equation*}
  \begin{aligned}
    \lambda(D)&=D\left(\int_{\Omega}|\nabla\phi_{D}|^{2}+\int_{\partial\Omega}\beta\phi_{D}^{2}\right)
    +\int_{\Omega}\left(-2\alpha\phi_{D}\nabla\phi_{D}\cdot\nabla
    m+\frac{\alpha^{2}}{D}\phi_{D}^{2}|\nabla
    m|^{2}\right)+\int_{\Omega}V\phi_{D}^{2}\\
    &\geq
    D\left(\int_{\Omega}|\nabla\phi_{D}|^{2}+\int_{\partial\Omega}\beta\phi_{D}^{2}\right)-\epsilon
    \int_{\Omega}|\nabla\phi_{D}|^{2}-C(\epsilon)\\
    &\geq
    (D-1)\mu_1+(1-\epsilon)\int_{\Omega}|\nabla\phi_{D}|^{2}-C\int_{\partial\Omega}\phi_{D}^{2}-C(\epsilon)\\
    &\geq
    (D-1)\mu_1+[1-(C+1)\epsilon]\int_{\Omega}|\nabla\phi_{D}|^{2}-C(\epsilon)\\
    &\geq (D-1)\mu_1-C(\epsilon),
  \end{aligned}
\end{equation*}
for any sufficiently small $\epsilon>0$. Letting $D\to +\infty$, we obtain $\lim_{D\to
  +\infty}\lambda(\alpha,D)=+\infty$.

\vskip8pt
We next prove (ii). Assume that $\mu_1<0$. Let
$\phi_{0}$ be the principal eigenfunction of (\ref{eq:4.2}) corresponding to $\mu_1$. We have
\[
  \int_{\Omega}|\nabla\phi_{0}|^{2}+\int_{\partial\Omega}\beta\phi_{0}^{2}=\mu_1<0.
\]
Thus
\begin{equation*}
  \begin{aligned}
    \lambda(D)&\leq
    D\left(\int_{\Omega}|\nabla\phi_{0}|^{2}+\int_{\partial\Omega}\beta\phi_{0}^{2}\right)
    +\int_{\Omega}\left(-2\alpha\phi_{0}\nabla\phi_{0}\cdot\nabla
    m+\frac{\alpha^{2}}{D}\phi_{0}^{2}|\nabla
    m|^{2}\right)+\int_{\Omega}V\phi_{0}^{2}\\
    &\leq D\mu_1 -C \to -\infty,\ \ \hbox{ as }D\to+\infty,
  \end{aligned}
\end{equation*}
and (ii) follows.

\vskip8pt Lastly we are going to prove (iii). If $\mu_1=0$, then
\[
  \int_{\Omega}|\nabla\phi_{0}|^{2}+\int_{\partial\Omega}\beta\phi_{0}^{2}=0.
\]
We may assume that
$\int_{\Omega}\phi_{0}^{2}=1$.
Choosing $\phi_{0}$ as a test function, we have
\begin{equation*}
  \begin{aligned}
    \lambda(D)&\leq
    D\left(\int_{\Omega}|\nabla\phi_{0}|^{2}+\int_{\partial\Omega}\beta\phi_{0}^{2}\right)
    +\int_{\Omega}\left(-2\alpha\phi_{0}\nabla\phi_{0}\cdot\nabla
    m+\frac{\alpha^{2}}{D}\phi_{0}^{2}|\nabla
    m|^{2}\right)+\int_{\Omega}V\phi_{0}^{2}\\
    &=\int_{\Omega}\left(-2\alpha\phi_{0}\nabla\phi_{0}\cdot\nabla
    m+\frac{\alpha^{2}}{D}\phi_{0}^{2}|\nabla
    m|^{2}\right)+\int_{\Omega}V\phi_{0}^{2}.
  \end{aligned}
\end{equation*}
This implies that $\lambda(D)\leq M$ for some positive constant $M$ independent of
$D\geq1$. In particular, it holds
 \begin{equation}
  \limsup_{D\to +\infty}\lambda(D)\leq
  \int_{\Omega}\left(V\phi_{0}^{2}-2\nabla m\cdot\nabla\phi_{0}\right).
\label{ld-a}
 \end{equation}
Let $\phi_{D}$ be the principal eigenfunction of (\ref{eq:4.1}) with
$\int_{\Omega}\phi_{D}^{2}=1$. Clearly,
\[
\int_{\Omega}|\nabla
\phi_{D}|^{2}+\int_{\partial\Omega}\beta\phi_{D}^{2}\geq 0.
\]
By means of Lemma \ref{lemma4.1}, it then follows that
\begin{equation*}
  \begin{aligned}
    \lambda(D)&=
    D\left(\int_{\Omega}|\nabla\phi_{D}|^{2}+\int_{\partial\Omega}\beta\phi_{D}^{2}\right)
    +\int_{\Omega}\left(-2\alpha\phi_{D}\nabla\phi_{D}\cdot\nabla
    m+\frac{\alpha^{2}}{D}\phi_{0}^{2}|\nabla
    m|^{2}\right)+\int_{\Omega}V\phi_{0}^{2}\\
    &\geq
    D\left(\int_{\Omega}|\nabla\phi_{D}|^{2}-C\int_{\partial\Omega}\phi_{D}^{2}\right)+\int_{\Omega}-2\alpha\phi_{0}\nabla\phi_{0}\cdot\nabla
    m-C\\
    &\geq D\left((1-C\epsilon)\int_{\Omega}|\nabla\phi_{D}|^{2}-C(\epsilon)\right)-\int_{\Omega}|\nabla\phi_{D}|^{2}-C\\
    &\geq
    D\left((1-C\epsilon)\int_{\Omega}|\nabla\phi_{D}|^{2}-C(\epsilon)\right)-C,\ \ \forall D\geq1,
  \end{aligned}
\end{equation*}
which implies that $\int_{\Omega}|\nabla \phi_{D}|^{2}\leq C$ for
some positive constant $C$ independent of $D\geq1$.

We claim that $\phi_{D}\to \phi_{0}$ weakly in $W^{1,2}(\Omega)$ and
strongly in $L^{2}(\Omega)$. Indeed, since $\phi_{D}$ are uniformly
bounded in $W^{1,2}(\Omega)$, there exists a subsequence of
$\{\phi_{D}\}$, still labelled by itself for convenience, and a function $\phi\in W^{1,2}(\Omega)$ such that
$\phi_{D}\to \phi$ weakly in $W^{1,2}(\Omega)$ and strongly in
$L^{2}(\Omega)$ as $D\to+\infty$. It is easy to see that $\phi$ satisfying $\phi\geq0$ a.e. in $\Omega$ is a weak (and then a classical) solution to
\begin{equation*}
\begin{cases}
  -\Delta \phi=0 &\hbox{ in }\Omega,\\
  \frac{\partial\phi}{\partial n}+\beta(x)\phi=0 &\hbox{ on }\partial\Omega,
\end{cases}
\end{equation*}
and $\int_{\Omega}\phi^{2}=1$. Due to the uniqueness of principal eigenfunction (up to multiplication) corresponding to the principal eigenvalue $\mu_1=0$, we can infer that $\phi=\phi_{0}$ and the claim is proved.

Therefore, we obtain
\begin{equation*}
  \begin{aligned}
    &\liminf_{D\to +\infty} \lambda(D)\\&= \liminf_{D\to
      +\infty}\left\{
      D\left(\int_{\Omega}|\nabla\phi_{D}|^{2}+\int_{\partial\Omega}\beta\phi_{D}^{2}\right)+\int_{\Omega}\left(-2\alpha\phi_{D}\nabla\phi_{D}\cdot\nabla
      m+\frac{\alpha^{2}}{D}\phi_{0}^{2}|\nabla
      m|^{2}\right)+\int_{\Omega}V\phi_{0}^{2}\right\}\\
    &\geq\lim_{D\to +\infty}\int_{\Omega}
    -2\alpha\phi_{0}\nabla\phi_{0}\cdot\nabla
    m+\int_{\Omega}V\phi_{0}^{2}\\
    &= \int_{\Omega}V\phi_{0}^{2}-2\alpha\nabla
    m\cdot\nabla\phi_{0},
\end{aligned}
\end{equation*}
which, together with \eqref{ld-a}, completes the proof of (iii).

\vskip8pt It is easily seen that {\rm (i)} holds if $\beta\geq,\not\equiv0$ and {\rm (iii)} holds if $\beta\equiv0$. It remains to
show $\mu_1<0$ and in turn (ii) holds when either $\int_{\partial\Omega}\beta<0$ or $\int_{\partial\Omega}\beta=0$ and $\beta\not\equiv0$.

Let $\phi_{0}$ be the principal eigenfunction of (\ref{eq:4.2}) corresponding to $\mu_1$. Then $\phi_{0}>0$ on $\bar\Omega$. Dividing the equation in (\ref{eq:4.2}) by $\phi_{0}$ and integrating the resulting equation by parts, we deduce
 \begin{equation}\label{ld-b}
 \mu_1|\Omega|=-\int_\Omega\frac{\Delta\phi_{0}}{\phi_{0}}=-\int_{\partial\Omega}\frac{\partial\phi_{0}}{\partial n}\cdot\frac{1}{\phi_{0}}-\int_\Omega\frac{|\nabla\phi_{0}|^2}{\phi_{0}^2}=\int_{\partial\Omega}\beta-\int_\Omega\frac{|\nabla\phi_{0}|^2}{\phi_{0}^2}.
 \end{equation}
Thus, $\mu_1\leq0$ if $\int_{\partial\Omega}\beta\leq0$. In particular, $\mu_1<0$ provided that $\int_{\partial\Omega}\beta<0$.

If $\int_{\partial\Omega}\beta=0$ and $\beta\not\equiv0$, we suppose that $\mu_1=0$. It then follows from \eqref{ld-b} that
$\phi_{0}$ must be a positive constant. Using the boundary condition in (\ref{eq:4.2}), we see that $\beta\equiv0$, arriving at a contradiction. Hence, $\mu_1<0$ holds under our assumption. The proof of Theorem \ref{th1.3} is now complete.
{\hfill $\Box$}

\begin{remark}\label{r3} We point out that the result similar to Theorem \ref{th1.3} holds for a more general eigenvalue problem, such as the following eigenvalue problem
\begin{equation}\label{general-p-a}
\begin{cases}
  -D\Delta \phi -2\alpha\nabla m(x)\cdot \nabla\phi+V(x)\phi=\mu\phi &\hbox{ in }\Omega,\\
  \phi=0& \hbox{ on } \Gamma_1\\
  \frac{\partial\phi}{\partial n}+\beta(x)\phi=0 &\hbox{ on
  }\Gamma_2,
\end{cases}
\end{equation}
where $\Gamma_1\cup\Gamma_2=\partial\Omega,\ \Gamma_1\cap\Gamma_2=\emptyset$, and $\beta\in C(\Gamma_2)$ is a given nonnegative function.
  \end{remark}

In one space dimension, we can improve Theorem \ref{th1.3}. By taking $\Omega=(0,1)$ without loss of generality,
we are led to consider the eigenvalue problem
\begin{equation}
  \label{eq:4.3}
\begin{cases}
  -D \phi_{xx} -2\alpha m_x\phi_x+V\phi=\lambda\phi,\ 0<x<1,\\
  -\phi_x(0)+k_0\phi(0)=0,\ \phi_x(1)+k_1\phi(1)=0.
\end{cases}
\end{equation}
Indeed, we have

\begin{cor}\label{cor} The following assertions hold.
  \begin{itemize}
\item[\rm{(i)}] If $k_0>-1$ and $ k_0+k_1+k_0k_1>0$, then
  $\lim_{D\to+\infty}\lambda(D)=+\infty$.

\item[\rm{(ii)}]  If either $k_0>-1$ and $k_0+k_1+k_0k_1<0$ or $k_0\leq-1$, then
  $\lim_{D\to+\infty}\lambda(D)=-\infty$.

\item[\rm{(iii)}]  If $k_0>-1$ and $k_0+k_1+k_0k_1=0$, then
  \[
    \lim_{D\to+\infty}\lambda(D)=\int_{0}^{1}\Big(V(x)\phi_{0}^{2}(x)-2\alpha
    m_x(x)({\phi_{0}})_x(x)\Big)dx,
  \]
  where $\phi_{0}$ is the solution to
\begin{equation}
  \label{eq:4.3}
\begin{cases}
  -\phi_{xx}=0,\ \ \phi(x)>0,\ \  0<x<1,\\
  -\phi_{x}(0)+k_0\phi(0)=0,\ \phi_{x}(1)+k_1\phi(1)=0,\\
  \int_{0}^{1}\phi^{2}(x)dx=1.
\end{cases}
\end{equation}

  \end{itemize}
\end{cor}

\begin{proof} We first consider the limiting problem \eqref{eq:4.3}. Assume that $\phi$ solves \eqref{eq:4.3}. Then
 $0$ is the principal eigenvalue of the following eigenvalue problem
 \begin{equation}
  \label{eq:4.3a}
\begin{cases}
  -\phi_{xx}=\mu\phi,\ \  0<x<1,\\
  -\phi_{x}(0)+k_0\phi(0)=0,\ \phi_{x}(1)+k_1\phi(1)=0.
\end{cases}
\end{equation}

By the Hopf boundary lemma, it is easily seen that
$\phi>0$ on $[0,1]$. From the first equation in \eqref{eq:4.3a}, it further follows that
$\phi(x)=ax+b$ for $x\in[0,1]$. Thus, $\phi(0)=b>0$ and $\phi(1)=a+b>0$. Clearly, if
 \begin{equation}\label{con1}
 b>0,\ \ \ a+b>0,
 \end{equation}
then $\phi(x)=ax+b>0$ for all $x\in[0,1]$. Thanks to the boundary condition, we have $-a+k_0b=0$ and $a+k_1(a+b)=0$, equivalently,
 \begin{equation}\label{con2}
 a=k_0b,\ \ \ k_0+k_1+k_0k_1=0.
 \end{equation}

If $a=0$, by \eqref{con1} and \eqref{con2}, obviously $k_0=k_1=0$. If $a>0$, then
it follows from \eqref{con1} and \eqref{con2} that $k_0>0$ and $a+b=a(1+\frac{1}{k_0})>0$, which becomes equivalent to $k_0>0$.
If $a<0$, in order to satisfy \eqref{con1} and \eqref{con2}, it is necessary that $k_0<0$ and $a+b=a(1+\frac{1}{k_0})>0$, that is,
$-1<k_0<0$.

The above analysis shows that if $0$ is the principal eigenvalue of the eigenvalue problem \eqref{eq:4.3a}, then
$k_0$ and $k_1$ must satisfy
  \begin{equation}\label{con3}
 k_0>-1,\ \  \ k_0+k_1+k_0k_1=0.
 \end{equation}
It can be also seen from the above analysis that $0$ is the principal eigenvalue once the condition \eqref{con3} is fulfilled.
That is, $0$ is the principal eigenvalue of \eqref{eq:4.3a} if and only if \eqref{con3} holds.

Denote by the principal eigenvalue $\mu_1(k_0,k_1)$ of \eqref{eq:4.3a}. Clearly, $\mu_1(k_0,k_1)$ depends continuously on the parameters $k_0$ and $k_1$, and is nondecreasing in $k_0,\,k_1\in\mathbb{R}$. One further observes that $\mu_1(k_0,k_1)>\mu_1(0,0)=0$ if $k_0,\,k_1>0$. These facts, combined with the previous analysis, enable us to assert that $\mu(k_0,k_1)>0$ if $k_0>-1$ and $ k_0+k_1+k_0k_1>0$ while $\mu_1(k_0,k_1)<0$ if either $k_0>-1$ and $k_0+k_1+k_0k_1<0$ or $k_0\leq-1$; one may refer to Figure 1.

As a consequence,  Corollary \ref{cor} follows by using Theorem \ref{th1.3}.
\end{proof}

\begin{figure}
\label{fig:a}
  \centering
  \includegraphics[width=5cm]{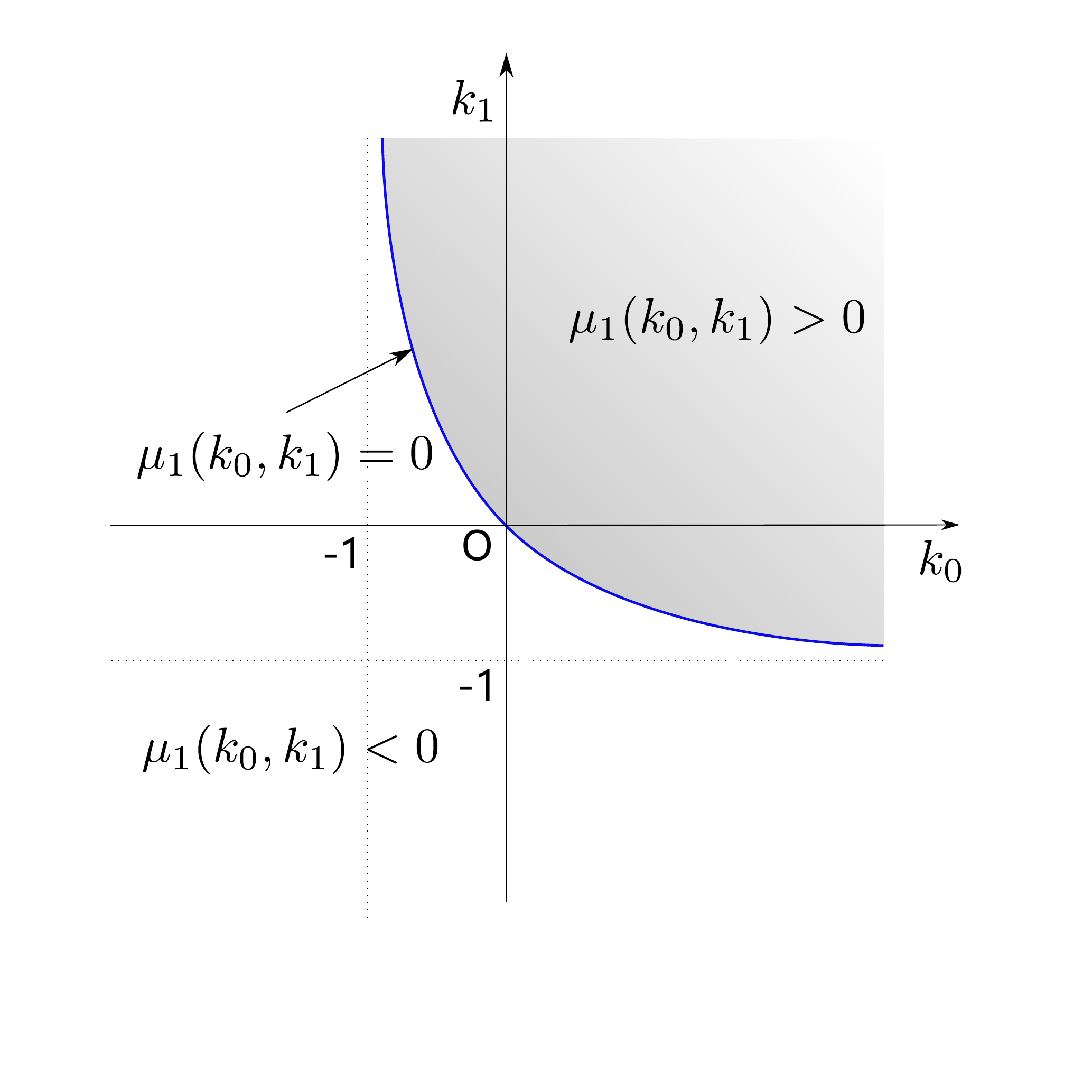}
  \caption{The principal eigenvalue $\mu_1(k_0,k_1)$ of \eqref{eq:4.3a}: \ $\mu_1(k_0,k_1)=0$ on the curve $\{(k_0,k_1)\in\mathbb{R}^2:\ k_0+k_1+k_0k_1=0,\ k_0>-1\}$;
  $\mu_1(k_0,k_1)>0$ when $(k_0,k_1)$ lies above the curve, and $\mu_1(k_0,k_1)<0$ when $(k_0,k_1)$ lies below the curve.}
\end{figure}

For the following eigenvalue problem
\begin{equation}
  \label{eq:4.4}
\begin{cases}
  -D \phi_{xx}-2\alpha m_{x}\phi_{x}+V\phi=\lambda\phi,\ 0<x<1,\\
  -\phi_{x}(0)+k_0\phi(0)=0,\ \phi(1)=0,
\end{cases}
\end{equation}
one can employ a similar but simpler argument as in Corollary \ref{cor} to deduce

\begin{cor}\label{cor1} The following assertions hold.
  \begin{itemize}
\item[\rm{(i)}] If $k_0>-1$, then $\lim_{D\to+\infty}\lambda(D)=+\infty$.

\item[\rm{(ii)}] If $k_0<-1$, then   $\lim_{D\to+\infty}\lambda(D)=-\infty$.

\item[\rm{(iii)}] If $k_0=-1$, then
  \[
    \lim_{D\to+\infty}\lambda(D)=\int_{0}^{1}\Big(V(x)\phi_{0}^{2}(x)-2\alpha
    m_{x}(x)(\phi_{0})_{x}(x)\Big)dx,
  \]
  where $\phi_{0}$ is the solution to
\begin{equation}
  \nonumber
\begin{cases}
  -\phi_{xx}=0,\ \ \phi(x)>0,\ \  0<x<1,\\
  -\phi_{x}(0)+k_0\phi(0)=0,\ \phi(1)=0,\\
  \int_{0}^{1}\phi^{2}(x)dx=1.
\end{cases}
\end{equation}

  \end{itemize}
\end{cor}

\vskip4pt
A result parallel to Corollary \ref{cor1} holds for the following eigenvalue problem
 \begin{equation}
  \label{eq:4.5}
\begin{cases}
  -D \phi_{xx} -2\alpha m_{x}\phi_{x}+V\phi=\lambda\phi,\ 0<x<1,\\
  \phi(0)=0,\ \phi_{x}(1)+k_1\phi(1)=0.
\end{cases}
\end{equation}

\section{An application to a reaction-diffusion-advection equation in a stream}

More recently, there is growing interest in modeling and understanding spatial population
dynamics in advective environments, i.e., environments where individuals are
exposed to unidirectional flow or biased dispersal; one may see \cite{LL,LLM,VL2} and the references therein.
The following reaction-diffusion-advection equation was used in \cite{JHSL,LPL,MJJL} to describe the
dynamics of a single species living in a spatially heterogeneous stream:
\begin{equation}\label{5.1}
  u_t-[Du_x-q(x)u]_x=r(x)u-u^2,\ \ t>0,\ 0<x<1,
\end{equation}
where the unknown function $u(t,x)$ is the density of the species at the time $t$ and location $x$,
the positive constant $D$ stands for the diffusive rate, the function $r\in C([0,1])$ represents the intrinsic growth rate (or the quality of the habitat)
and the species population will grow for $r>0$ while decline for $r<0$. The nonnegative function $q\in C^1([0,1])$ accounts for the
advection, pointing towards larger $x\geq0$, which therefore implies that $x=0$ is the upstream end of the stream and $x=1$
is the downstream end. It is now widely recognized that
variations in the stream flow are critically important for the ecosystem integrity of
riverine environments; see, for example, \cite{BA,PAB}. Thus, we assume that $q$ is a function of the spatial variation $x$.
In particular, the function $q$ may vanish somewhere the habitat $[0,1]$, which can reflect the existence of {\it buffer zone} in the stream.

Boundary conditions (B.C) should be prescribed at the upstream end $x=0$
and the downstream end $x=1$. Since  the upstream boundary $x=0$ is the stream surface and usually no individuals will pass through,
it is natural to impose no flux boundary condition there; that is, we have
 \begin{equation}\label{5.2}
 Du_x(t,0)-q(0)u(t,0)=0,\ \ t>0.
 \end{equation}

As for the B.C. at the downstream end $x=1$, there are several possible choices motivated by different ecological
scenarios as discussed in \cite{LL}. In the following, we shall only consider three types of B.C.

\vskip8pt
{\it Type 1:\ No-flux B.C.}\ \ Gravity pulls algae in a lake or ocean towards the bottom (advection), whereas buoyancy
allows for upward movement (diffusion) \cite{HAES}. The upstream boundary is the water surface, the downstream boundary
is the ground, and so no flux crosses the downstream end $x=1$. This leads us to assume that
\begin{equation}\label{5.2-a}
 Du_x(t,1)-q(1)u(t,1)=0,\ \ t>0.
 \end{equation}
Together with \eqref{5.2}, we refer to this situation as no-flux/no-flux, or NF/NF for short.

\vskip8pt
{\it Type 2:\ Free-flux B.C.}\ \ When the stream flows
into a freshwater lake, individuals can enter the downstream end of the stream from
the lake by diffusion. The flux into the lake is only the advective flux, the diffusive
flux into and from the lake balances \cite{LL,VL1}. Hence, we have
the following downstream condition
\begin{equation}\label{5.2-b}
 u_x(t,1)=0,\ \ t>0.
 \end{equation}
We refer to this and \eqref{5.2} as no-flux/free-flow or NF/FF for short.

\vskip8pt
{\it Type 3:\ Hostile B.C.}\ \ When individuals do not return into the patch after leaving at the downstream end, we
obtain the hostile downstream condition:
 \begin{equation}\label{5.2-c}
 u(t,1)=0,\ \ t>0,
 \end{equation}
We refer to this and \eqref{5.2} as no-flux/hostile or NF/H for short.
For example, most freshwater organisms die when they reach the ocean. Such a downstream
condition was originally proposed in \cite{SG}.

\vskip8pt Given a continuous initial datum $u(0,x)=u_0(x)\geq,\not\equiv0$ on $[0,1]$,
under one of the three types of B.C. mentioned above, it is well known that \eqref{5.1} admits a unique solution $u(t,x)$,
which exists for all time $t>0$ and $u(t,x)>0$ for all $t>0,\,0<x<1$.
It is also a standard fact that the long-time behavior of the single species governed
by \eqref{5.1} is determined by the sign of the principal eigenvalue to the following eigenvalue problem
(linearized at zero of problem \eqref{5.1}):
 \begin{equation}\label{5.3}
  -[D\psi_x-q(x)\psi]_x-r(x)\psi=\lambda\psi,\ \ \ 0<x<1,
\end{equation}
subject to the NF/NF or NF/FF or NF/H B.C. In any B.C. case, we always have

\begin{proposition}\label{prop4.1} Let $\lambda(D)$ denote the principal eigenvalue of \eqref{5.3}. The following assertions hold.
  \begin{itemize}
\item[\rm{(i)}] If $\lambda(D)\geq0$, then $u(t,x)\to0$ uniformly on $[0,1]$ as $t\to\infty$.

\item[\rm{(ii)}] If $\lambda(D)<0$, then $u(t,x)\to U(x)$ uniformly on $[0,1]$ as $t\to\infty$, where $U$ is the unique positive steady state solution of \eqref{5.1}.

  \end{itemize}

\end{proposition}

Set $\psi(x)=\phi(x)\exp\big({\frac{1}{D}\int_0^x q(s)ds}\big)$. Then \eqref{5.3} coupled with the NF/NF B.C (corresponding to \eqref{5.2} and \eqref{5.2-a}) is reduced to the following
 \begin{equation}\label{5.4-1}
  -D\phi_{xx}-q(x)\phi_x-r(x)\phi=\lambda\phi,\ \ 0<x<1;\ \ \ \phi_x(0)=\phi_x(1)=0,
\end{equation}
and \eqref{5.3} coupled with the NF/FF B.C (corresponding to \eqref{5.2} and \eqref{5.2-b}) becomes
 \begin{equation}\label{5.4-2}
  -D\phi_{xx}-q(x)\phi_x-r(x)\phi=\lambda\phi,\ \ 0<x<1;\ \ \ \phi_x(0)=\phi_x(1)+\frac{q(1)}{D}\phi(1)=0,
\end{equation}
and  \eqref{5.3} coupled with the NF/H B.C (corresponding to \eqref{5.2} and \eqref{5.2-c}) becomes
 \begin{equation}\label{5.4-3}
  -D\phi_{xx}-q(x)\phi_x-r(x)\phi=\lambda\phi,\ \ 0<x<1;\ \ \ \phi_x(0)=\phi(1)=0.
\end{equation}

For sake of clarity, instead of $\lambda(D)$, we use $\lambda(NF/NF)$, $\lambda(NF/FF)$ and $\lambda(NF/H)$ to denote the principal eigenvalue of \eqref{5.4-1} \eqref{5.4-2} and \eqref{5.4-3}, respectively.

In what follows, we shall consider four types of the stream flow $q$, depending on how $q$ vanishes (if it happens) on the habitat $[0,1]$; we call the vanishing region of $q$ as the {\it buffer} since the flow speed/velocity there is zero:

\vskip8pt

  \begin{itemize}
\item[\rm{ Case (a)}:]\ $q>0\ \mbox{on}\ [0,1]$. No buffer exists.
\vskip3pt
\item[\rm{ Case (b)}:]\ $q=0\ \mbox{on}\ [0,x_0]$ for some $0<x_0<1$ and $q>0\ \mbox{on}\ (x_0,1]$:\ \ $[0,x_0]$ is the buffer.
\vskip3pt
\item[\rm{ Case (c)}:]\ $q=0\ \mbox{on}\ [x_1,x_2]$ for some $0<x_1<x_2<1$ and $q>0\ \mbox{on}\ [0,x_1)\cup(x_2,1]$:\ \ $[x_1,x_2]$ is the buffer.
\vskip3pt
\item[\rm{ Case (d)}:]\ $q=0\ \mbox{on}\ [x_0,1]$ for some $0<x_0<1$ and $q>0 \ \mbox{on}\ [0,x_0)$:\ \ $[x_0,1]$ is the buffer.

  \end{itemize}

In each case mentioned above, due to Corollaries \ref{cor} and \ref{cor1}, we have
$$
 \lim_{D\to\infty}\lambda(NF/NF)=-\int_0^1 r(x)dx,\ \  \lim_{D\to\infty}\lambda(NF/H)=\infty.
 $$
By some simple elliptic compactness analysis, one can also easily see that
 $$
 \lim_{D\to\infty}\lambda(NF/FF)=-\int_0^1 r(x)dx.
 $$

In the sequel, for convenience we refer to a species who moves/migrates fast (i.e., the diffusion rate $D$ is large) as a faster species, and a species who moves/migrates slowly (i.e., $D$ is small) as a slower species.

In light of Proposition \ref{prop4.1}, the above results imply biologically that, regardless of whether a buffer zone exists or not, when the downstream end belongs to a hostile environment, a faster species will eventually die out, while a faster species may persist in the long run even when the downstream end satisfies a flux-free or a flow free condition.

In what follows, we are concerned with the effect of a buffer zone on the persistence/extinction of a slower species.
Making use of Remark \ref{r-N}, Theorem \ref{th1.2} and Remark \ref{r2.2}, we are able to state

\begin{proposition}\label{prop4.2} The following assertions hold.

  \begin{itemize}
\item[\rm{(1)}] Assume that {\rm case (a)} holds, we have
 $$
 \lim_{D\to0}\lambda(NF/NF)=-r(1),\ \lim_{D\to0}\lambda(NF/FF)=\lim_{D\to0}\lambda(NF/H)=\infty.
 $$

\item[\rm{(2)}] Assume that {\rm case (b)} holds, we have
 $$
 \lim_{D\to0}\lambda(NF/NF)=\min\big\{\min_{x\in[0,x_0]}(-r(x)),\,-r(1)\big\},
 $$
 $$
 \lim_{D\to0}\lambda(NF/FF)=\lim_{D\to0}\lambda(NF/H)=\min_{x\in[0,x_0]}(-r(x)).
 $$

\item[\rm{(3)}] Assume that {\rm case (c)} holds, we have
 $$
 \lim_{D\to0}\lambda(NF/NF)=\min\big\{-r(1),\,\min_{x\in[x_1,x_2]}(-r(x))\big\},$$
 $$
 \lim_{D\to0}\lambda(NF/FF)=\lim_{D\to0}\lambda(NF/H)=\min_{x\in[x_1,x_2]}(-r(x)).
 $$

\item[\rm{(4)}] Assume that {\rm case (d)} holds, we have
 $$
 \lim_{D\to0}\lambda(NF/NF)=
 \lim_{D\to0}\lambda(NF/FF)=\lim_{D\to0}\lambda(NF/H)=\min_{x\in[x_0,1]}(-r(x)).
 $$
  \end{itemize}

\end{proposition}

From the viewpoint of ecological evolution, Propositions \ref{prop4.1} and \ref{prop4.2} suggest that if no buffer exits,
a slower species will become extinct eventually when the downstream end satisfies the free-flow or hostile B.C, and it may survive when
the downstream end satisfies a NF B.C. In sharp contrast, if a buffer exits, a slower
species may persist in any B.C. case; in other words, a buffer may contribute to the species persistence. Nevertheless, whether
a buffer is indeed helpful to the species persistence relies on the sign of the maximum value of the growth rate function $r$ over the buffer for the NF/FF or NF/H B.C., and the sign of the maximum value of the growth rate function $r$ over the buffer and the downstream end for the NF/NF B.C.


\vskip20pt

\begin{thebibliography}{}
\bibliographystyle{siam}
\setlength{\baselineskip}{12pt}

%\bibitem{AMR} R. Abraham, J.E. Marsden, T. Ratiu, Manifolds, tensor
%analysis, and applications, 2nd. Vol. 75. Applied Mathematical
%Sciences. New York: Springer-Verlag, 1988, pp. x+654.


%\bibitem{ALo}
%D. Aleja, J. L{\'o}pez-G{\'o}mez, Some paradoxical effects of the
%advection on a class of diffusive equations in ecology, Discrete
%Contin. Dyn. Syst. B {\bf 19}(2014), 3031-3056.


\bibitem{BHN}
H. Berestycki, F. Hamel, N. Nadirashvili, Elliptic eigenvalue
problems with large drift and applications to nonlinear propagation
phenomena, Comm. Math. Phys. 253(2005), 451-480.

%\bibitem{BNV}
%H. Berestycki, L. Nirenberg, S.R.S. Varadhan, The principal
%eigenvalue and maximum principle for second-order elliptic operators
%in general domains, Comm. Pure Appl. Math. {\bf 47}(1994), 47-92.

%\bibitem{BR}
%H. Berestycki, L. Rossi, On the principal eigenvalue of elliptic
%operators in $\bR^N$ and applications, J. Eur. Math. Soc. {\bf
%8}(2006), 195-215.

%\bibitem{BD}
%D. Bucur, D. Daners, An alternative approach to the Faber-Krahn
%inequality for Robin problems, Cal. Var. Partial Differential
%Equations {\bf 37}(2010), 75-86.


\bibitem{BA}
S.E. Bunn, A.H. Arthington, Basic principles and ecological consequences of altered flow
regimes for aquatic biodiversity, Environ Manag. 30(2002), 492-507.


\bibitem{CL1}
X.F. Chen, Y. Lou, Principal eigenvalue and eigenfunctions of an
elliptic operator with large advection and its application to a
competition model, Indiana Univ. Math. J. 57(2008), 627-658.

\bibitem{CL2}
X.F. Chen, Y. Lou, Effects of diffusion and advection on the
smallest eigenvalue of an elliptic operator and their applications,
Indiana Univ. Math. J. 61(2012), 45-80.


%\bibitem{Dan}
%D. Daners, Eigenvalue problems for a cooperative system with a large parameter,
%Adv. Nonlinear Stud. {\bf 13}(2013), 137-148.

%\bibitem{Dan2}
%D. Daners, A Faber-Krahn inequality for Robin problems in any space
%dimension, Math. Ann. {\bf 335}(2006), 767-785.

%\bibitem{DKe}
%D. Daners, J. Kennedy, Uniqueness in the Faber-Krahn inequality for
%Robin problems, SIAM J. Math. Anal. {\bf 39}(2007), 1191-1207.


\bibitem{DF}
A. Devinatz, R. Ellis, A. Friedman, The asymptotic behavior of the
first real eigenvalue of the second-order elliptic operator with a
small parameter in the higher derivatives, II. Indiana Univ. Math.
J. 23(1973/74), 991-1011.

%\bibitem{DF}
%A. Devinatz, A. Friedman, Asymptotic behavior of the principal
%eigenfunction for a singularly perturbed Dirichlet problem, Indiana
%Univ. Math. J. 27(1978), 143-157.

\bibitem{Du} Y. Du, Order Structure and Topological Methods in Nonlinear Partial Differential Equations,
Vol. 1, Maximum Principles and Applications, World Scientific, Singapore, 2006.

%\bibitem{DH1} Y. Du, S.-B. Hsu, Concentration phenomena in a nonlocal quasi-linear
%problem modeling phytoplankton: I. Existence, SIAM J. Math. Anal. 40(2008), 1419-1440.

%\bibitem{DH2} Y. Du, S.-B. Hsu, Concentration phenomena in a nonlocal quasi-linear problem modeling phytoplankton: II. Limiting profile, SIAM J. Math. Anal. 40(2008), 1441-1470.

%\bibitem{DH3} Y. Du, S.-B. Hsu, On a nonlocal reaction-diffusion problem arising from the modeling of phytoplankton growth. SIAM J. Math. Anal. 42(2010), 1305-1333.

%\bibitem{DM} Y. Du, L. Mei, On a nonlocal reaction-diffusion-advection equation
%modelling phytoplankton dynamics, Nonlinearity 24(2011), 319-349.



\bibitem{Ec}
M. Eckhoff, Precise asymptotics of small eigenvalues of reversible diffusions
in the metastable regime, Ann. Probab. 33(2005), 244-299.

\bibitem{Evans}
L.C. Evans, Partial differential equations. Graduate Studies in Mathematics,
American Mathematical Society, Providence, RI, 1998. ISBN: 0-8218-0772-2

\bibitem{FK}
P. Freitas, D. Krej\v{c}i\v{r}\'{i}k, The first Robin eigenvalue with negative boundary parameter,
Adv. Math. 280(2015), 322-339.


%\bibitem{Fr1}
%A. Friedman,  {\it Partial Differential Equations of Parabolic
%Type}, Printice-Hall, Englewood Cliffs, N.J., 1964.

\bibitem{Fr}
A. Friedman, The asymptotic behavior of the first real eigenvalue
of a second order elliptic operator with a small parameter in the
highest derivatives, Indiana Univ. Math. J. 22(1973),
1005-1015.

\bibitem{Gr}
P. Grisvard, Elliptic problems in nonsmooth domains, volume 24 of Monographs and Studies in Mathematics.
Pitman (Advanced Publishing Program), Boston, MA, 1985.

%\bibitem{GT}
%D. Gilbarg, N.S. Trudinger, Elliptic Partial Differential Equation of Second Order,
%Springer, 2001.

%\bibitem{GN}
%D.S. Grebenkov, B.-T. Nyuyen, Geometrical structure of Laplacian eigenfunctions,
%SIAM Rev. {\bf 55}(2013), 601-667.

%\bibitem{He}
%P. Hess, Periodic-parabolic Boundary Value Problems and Positivity,
%Pitman Research Notes in Mathematics,
%vol. 247. Longman Sci. Tech., Harlow, 1991.

%\bibitem{HL}
%S.-B. Hsu, Y. Lou, Single phytoplankton species growth with light
%and advection in a water column, SIAM J. Appl. Math. 70(2010),
%2942-2974.

\bibitem{HAES}
J. Huisman, M. Array¨¢s, U. Ebert, B. Sommeijer, How do sinking phytoplankton species manage to
persist, Am. Nat. 159(2002), 245-254.

\bibitem{JHSL}
Y. Jin, F.M. Hilker, P.M. Steffler, M.A. Lewis, Seasonal invasion dynamics in a spatially
heterogeneous river with fluctuating flows, Bull. Math. Biol. 76(2014), 1522-1565.


%\bibitem{KOY}
%T. Kolokolnikov, C. Ou, Y. Yuan, Profiles of self-shading, sinking phytoplankton with finite depth,
%J. Math. Biol. 59(2009), 105-122.

\bibitem{LOS}
A.A. Lacey, J.R. Ockendon, J. Sabina, Multidimensional reaction diffusion equations with nonlinear boundary conditions,
SIAM J. Appl. Math. 58(1998), 1622-1647.


%\bibitem{Lam}
%K.-Y. Lam, Concentration phenomena of a semilinear elliptic equation
%with large advection in an ecological model, J. Differential Equations
%250(2011), 161-181.

%\bibitem{LLL}
%K.-Y. Lam, Y. Lou, F. Lutscher, The emergence of range limits in advective environments,
%SIAM J. Appl. Math. 76(2016), 641-662.


%\bibitem{LNi}
%K.-Y. Lam, W.-N. Ni, Advection-mediated competition in general environments,
%J. Differential Equations 257(2014), 3466-3500.

\bibitem{LL}
Y. Lou, F. Lutscher, Evolution of dispersal in open advective environments,
J. Math. Biol. 69(2014), 1319-1342.

\bibitem{LLM}
F. Lutscher, M.A. Lewis, E. McCauley, Effects of heterogeneity on spread and persistence in
rivers, Bull Math Biol. 68(2006), 2129-2160.

\bibitem{LPL}
F. Lutscher, E. Pachepsky, M.A. Lewis, The effect of dispersal patterns on stream populations,
SIAM Rev. 47(2005), 749-772.

\bibitem{MJJL}
H.W. McKenzie, Y. Jin, J. Jacobsen, M.A. Lewis, $R^0$ analysis of a spatiotemporal model for a stream
population, SIAM J. Appl. Dyn. Syst. 11(2012), 567-596.

\bibitem{PP}
K. Pankrashkin, N. Popoff, An effective Hamiltonian for the eigenvalue asymptotics of the Robin Laplacian with a large parameter,
J. Math. Pures Appl. 106(2016), 615-650.


\bibitem{PZ1}
R. Peng, X.-Q. Zhao,
A nonlocal and periodic reaction-diffusion-advection model of a single phytoplankton species,
J. Math. Biol. 72(2016), 755-791.

\bibitem{PZ2}
R. Peng, X.-Q. Zhao, Effects of diffusion and advection on the principal
eigenvalue of a periodic-parabolic problem with applications, Calc. Var.
Partial Differential Equations, 54(2015), 1611-1642.

\bibitem{PZh}
R. Peng, M. Zhou, Effects of large degenerate advection and boundary conditions on the principal eigenvalue and
its eigenfunction of a linear second order elliptic operator, Indiana Univ. Math. J. (2018), in press.

\bibitem{PAB}
N.L. Poff, J.D. Allan, M.B. Bain, J.R. Karr, K.L. Prestegaard, B.D. Richter, et al. The natural
flow regime: a paradigm for river conservation and restoration, BioScience 47(1997), 769-784.

\bibitem{SG}
D.C. Speirs, W.S.C. Gurney, Population persistence in rivers and estuaries, Ecology 82(2001), 1219-1237.

\bibitem{VL1}
O. Vasilyeva, F. Lutscher, Population dynamics in rivers: analysis of steady states, Can. Appl. Math. Q.
18(2011), 439-469.

\bibitem{VL2}
O. Vasilyeva, F. Lutscher,  Competition in advective environments, Bull. Math. Biol. 74(2012), 2935-2958.

\bibitem{Ve1}
A.D. Ventcel', The asymptotic behavior of the largest eigenvalue of a second order elliptic differential operator with a small parameter multiplying the highest derivatives. (Russian) Dokl. Akad. Nauk SSSR 202(1972), 19-22.

\bibitem{Ve2}
A.D. Ventcel', The asymptotic behavior of the eigenvalues of matrices with elements of the order $\exp\{-V_{ij} /(2\epsilon^2)\}$. (Russian)
Dokl. Akad. Nauk SSSR 202(1972), 263-265.

\bibitem{W}
A.D. Wentzell, On the asymptotic behavior of the first eigenvalue of
a second order differential operator with small parameter in higher
derivatives, Theory Prob. Appl. 20(1975), 599-602.
\end{thebibliography}
\end {document}